%% file: main.tex
\documentclass[a4paper]{amsart}

\input{packages}
\input{settings}

\author{Matthias Bärlin}
\author{Konrad Keßler}
\date{\today}
\title{Partial Regularity for $\A$-quasiconvex Functionals}

\begin{document}

\input{draft-final/abstract}
\maketitle
\input{draft-final/introduction}

\input{draft-final/preliminaries}
\input{draft-final/fubini-type-theorem}
\input{draft-final/b-harmonic-approximation}
\input{draft-final/excess-decay}

\clearpage
\printbibliography

\end{document}

%% file: packages.tex
\usepackage[T1]{fontenc}
\usepackage[utf8]{inputenc}
\usepackage[english]{babel}
\usepackage{csquotes}

\usepackage{amssymb} 
\usepackage{mathtools,amsmath,amsthm}
\usepackage{mathrsfs}

\usepackage[dvipsnames]{xcolor}		
\usepackage{enumitem}	
\usepackage{hyperref}	

\usepackage{lmodern}

\usepackage{microtype}

\makeatletter
\@ifpackageloaded{lmodern}{\usepackage{esint} \usepackage{amssymb}}{}
\@ifpackageloaded{mathpazo}{\usepackage{esint}}{}
\@ifpackagewith{mathdesign}{bitstream-charter}{\usepackage{esint}}{}
\makeatother

\usepackage[style=numeric,hyperref=true,sorting=nyt,backend=biber]{biblatex}
\DeclareNameAlias{author}{family-given}

%% file: settings.tex
\numberwithin{equation}{section}

\theoremstyle{definition}
\newtheorem{definition}{Definition}[section] 
\newtheorem*{definition*}{Definition}

\newtheorem*{example*}{Example}

\theoremstyle{plain}
\newtheorem{theorem}[definition]{Theorem}
\newtheorem*{theorem*}{Theorem}
\newtheorem{proposition}[definition]{Proposition}
\newtheorem*{proposition*}{Proposition}
\newtheorem{corollary}[definition]{Corollary}
\newtheorem*{corollary*}{Corollary}
\newtheorem{lemma}[definition]{Lemma}
\newtheorem*{lemma*}{Lemma}

\theoremstyle{remark}

\newtheorem*{remark*}{Remark}

\newcommand{\N}{\mathbb{N}}

\newcommand{\R}{\mathbb{R}}
\newcommand{\C}{\mathbb{C}}

\renewcommand{\d}{\,\mathrm{d}} 	
\newcommand{\BV}{\mathrm{BV}}
\newcommand{\BD}{\mathrm{BD}}
\newcommand{\BVA}{\mathrm{BV}^{\mathbb{A}}}
\newcommand{\SOB}{\mathrm{W}} 
\newcommand{\CON}{\mathrm{C}} 
\newcommand{\LEB}{\mathrm{L}} 
\newcommand{\A}{\mathbb{A}}

\newcommand{\B}{\mathrm{B}}
\newcommand{\loc}{\mathrm{loc}}
\newcommand{\spann}{\mathrm{span}}
\newcommand{\com}{\mathrm{c}} 
\newcommand{\excess}{\mathbf{E}} 
\newcommand{\excessn}{\wtilde{\mathbf{E}}} 
\newcommand{\norm}[1]{\left\lVert#1\right\rVert} 

\DeclareMathOperator{\Tr}{Tr}
\DeclareMathOperator{\Div}{div}

\newcommand{\restr}[1]
{
	\bigr\rvert_{#1}
}

\providecommand{\gtrless}{\mathrel{\substack{>\\[-.18em]<}}}

\newcommand{\mres}[1]{\mathbin{\vrule height 1.6ex depth 0pt width
0.13ex\vrule height 0.13ex depth 0pt width 1.3ex}_{\,#1}}

\newcommand{\overbar}[1]{{\mkern 3mu\overline{\mkern-3mu#1\mkern-3mu}\mkern 1mu}}
\makeatletter
\def\namedlabel#1#2{\begingroup
    #2%
    \def\@currentlabel{#2}%
    \phantomsection\label{#1}\endgroup
}
\makeatother
\hypersetup{
    colorlinks,
    citecolor=BurntOrange,
    filecolor=black,
    linkcolor=BurntOrange,
    urlcolor=black
}

\newcommand{\eg}{\hbox{e.\,g.}} 
\newcommand{\ie}{\hbox{i.\,e.}}

\renewcommand{\mathcal}[1]{\mathscr{#1}}

\providecommand{\PARENS}[1]{\left(#1\right)}
\providecommand{\SQRT}[1]{\sqrt{#1}}
\providecommand{\wtilde}[1]{\widetilde{#1}}

\providecommand{\ccases}[1]{\begin{cases}#1\end{cases}}
\providecommand{\slashint}{\fint}
\providecommand{\coloneq}{\coloneqq}
\providecommand{\mbf}[1]{\mathbf{#1}}

\providecommand{\lcbrace}{\lbrace}
\providecommand{\rcbrace}{\rbrace}

\makeatletter
\@ifpackageloaded{mtpro2}{}{
\newcommand*\rel@kern[1]{\kern#1\dimexpr\macc@kerna}
\newcommand*\wbar[1]{%
  \begingroup
  \def\mathaccent##1##2{%
    \rel@kern{0.8}
    \overline{\rel@kern{-0.8}\macc@nucleus\rel@kern{0.2}}%
    \rel@kern{-0.2}%
  }%
  \macc@depth\@ne
  \let\math@bgroup\@empty \let\math@egroup\macc@set@skewchar
  \mathsurround\z@ \frozen@everymath{\mathgroup\macc@group\relax}%
  \macc@set@skewchar\relax
  \let\mathaccentV\macc@nested@a
  \macc@nested@a\relax111{#1}%
  \endgroup
}
}
\makeatother

%% file: draft-final/abstract.tex
\begin{abstract}
We establish partial Hölder regularity for (local) generalised  minimisers of variational problems involving  strongly quasi-convex integrands of linear growth, where the full gradient is replaced by a first order homogeneous differential operator $\A$ with constant coefficients. Working under the assumption of $\A$ being $\C$-elliptic, this is achieved by adapting a method recently introduced in \autocite{GK2019, Gme2020}.
\end{abstract}

%% file: draft-final/introduction.tex
\section{Introduction}
\subsection{Variational Problems}
The analysis of functionals taking the form $\mathcal{F}=\int_{\Omega}f(\nabla u)\d x$ is a major task in the calculus of variations with a long standing tradition. Let us suppose that $\Omega\subset \R ^{n}$ is open and bounded and that $u$ is a weakly differentiable $\R ^{N}$-valued map. The growth condition on the integrand $f\in\CON (\R ^{N\times n})$ determines the functional analytic environment in which we analyse $\mathcal{F}$. A standard growth assumption, which has been studied intensively  in the field, constitutes the following: There exist $p\in [1,\infty )$ and a constant $L>0$ such that for all $z\in \R ^{N\times n}$ we have
\begin{equation}
    |f(z)|\leq L\left(1+|z|^{p}\right).
\end{equation}

The existence of minimisers within a given class of functions (or maps) is a fundamental question. Concretely, we want to minimise the functional $\mathcal{F}$ in $g+\SOB ^{1,p}_{0}(\Omega;\R ^{N})$ for a prescribed Dirichlet boundary datum $g\in \SOB ^{1,p}(\Omega;\R ^{N})$. In the super-linear growth case $p>1$, this task can immediately be tackled by means of the direct method, a lower semi-continuity method dating back to Tonelli. Consequently, the question arises under which assumptions on $f$ is the functional $\mathcal{F}$ sequential weakly lower semi-continuous in $\SOB ^{1,p}(\Omega;\R ^{N})$. Towards this question, convexity certainly suffices, but in the vectorial case $(N>1)$ it is seen to not be a necessary condition. Ball and Murat \autocite{BM1984} have shown that Morrey's notion of quasi-convexity \autocite{Mo1952}, \ie, for all $z\in \R^{N\times n}$ and all $\zeta \in \CON ^{1}_{\com}(\mathbb{B};\R ^{N})$ we have
$$ f\left(z\right)\leq \fint_{\mathbb{B}}f\left(z+\nabla \zeta\right)\d x,$$
turns out to be necessary. In fact, for unsigned integrands $f$, quasi-convexity is also a sufficient condition, as \citeauthor{AF1984} have shown in \autocite{AF1984}. 

Once having addressed the issue of existence of minimisers, we would like to know what further information about a minimiser we can extract. This question dates back to David Hilbert \autocite{Hil1900} and is today known by the name of regularity theory in the calculus of variations. There are many different notions of regularity and in our setting, we are interested whether a minimiser is (locally) of the class $\CON^{1,\alpha}$. In the scalar case $(N=1)$ the notions of convexity and quasi-convexity coincide and the regularity theory, at least in the quadratic growth case, reduces to the regularity of solutions of elliptic equations established by De Giorgi \autocite{Gio1957}, Nash \autocite{Nash1958} and Moser \autocite{Mo1960}. However, in the vectorial case, the regularity of minimisers can no longer be extracted from the Euler-Lagrange equation only, because, as various counterexamples show \autocite{Gio1968,Nec1977,Fr1970}, there is no such theory for elliptic systems in general. Furthermore, full Hölder regularity can no longer be expected since minimisers may be unbounded within a small set. Adapting ideas from geometric measure theory developed by Almgren \autocite{Alm1968} and De Giorgi \autocite{Gio1961}, Evans established in the non-parametric setting a fundamental \textit{partial} regularity result assuming a stronger notion of quasi-convexity \autocite{Eva1986}. This means that a minimiser enjoys Hölder regularity outside a small set.  We stress that partial regularity is a feature of the vectorial case. After the quadratic case, the super-quadratic case ($p\geq 2$) was established by Acerbi and Fusco in \autocite{AF1987}. Later on, the sub-quadratic case ($1<p\leq 2$) was resolved partially by Carozza and Passarelli di Napoli in \autocite{CP1996} and then fully by \citeauthor{CFM1998} in \autocite{CFM1998}. Some time later, even the Orlicz growth case has been resolved by \citeauthor{DLSV2012} in \autocite{DLSV2012}. An overview of related results can be found in \autocite{Min2006, Min2008, Bec2016, Giu2003}. The case of linear growth for quasi-convex integrands, however, had remained an open problem, since the classical methods were bound to fail due to the lack of weak compactness. Only in the recent years, partial local Hölder regularity of the (distributional) gradient of (local) $\BV$-minimisers in the quasi-convex setting has been established by \citeauthor{GK2019} in \autocite{GK2019}.

Let us try to roughly describe the underlying ideas on how to obtain partial Hölder regularity of the weak gradient of a minimiser $u$ of $\mathcal{F}$: The main objective is to prove a decay estimate for the excess of $u$. The excess is a quantity that, similar to Campanato's semi-norm, measures by means of integrals the rate of oscillation of $\nabla u$. The goal is to show that if the excess is small enough, it decays with any rate $\alpha \in (0,1)$. By means of a Caccioppoli Inequality, one passes from the excess, which depends on the gradient, to a quantity depending merely on $u$. The Caccioppoli Inequality, in turn, builds on the combination of minimality and a stronger notion of quasi-convexity by means of Widman's hole-filling trick. On the level of order $0$, the strategey then is to approximate the minimiser by a $\mathcal{B}$-harmonic map $h$, where $\mathcal{B}$ denotes a strongly Legendre-Hadamard elliptic bilinear form on $\R^{N\times n}$. The map $h$ has a good decay since it solves a homogeneous elliptic system. The difficulty is to construct $h$ in such a way that $u-h$ has good decay as well. Classically, the construction of $h$ follows an indirect approach utilising compactness, for example, of the embedding $\SOB^{1,2}(\mathbb{B})\to \LEB^{2}(\mathbb{B})$. This kind of approximation goes by the name of $\mathcal{B}$-\textit{Harmonic Approximation Lemma} \autocite{Gio1961,DG2000,DM2009}.

Many generalisations of the functional $\mathcal{F}$ have been studied. Here, we are going to replace the full gradient with a first order homogeneous differential operator $\A$. The two most prominent examples are the symmteric gradient $\varepsilon$ and the trace-free symmetric gradient $\tilde{\varepsilon}$. Let $V$ and $W$ be two real and finite dimensional Hilbert spaces. We are going to consider differential operators of the form
$$ \A=\sum\limits_{\alpha =1}^{n}\A _{\alpha}\partial _{\alpha}, \ \A_{\alpha}\in \mathcal{L}(V;W).$$
For $ \xi\in \mathbb{K}^{n}$, $\mathbb{K}=\R, \C$, the  linear map $\A[\xi]v=\sum_{\alpha =1}^{n}\xi _{\alpha}\A _{\alpha}v$, modulo a factor of $-i$, is called the \textit{symbol map} associated to the differential operator $\A$. We say that $\A$ is $\mathbb{K}$-\textit{elliptic} if the symbol map $\A[\xi]$ is \textit{one-to-one} for all $\xi\in \mathbb{K}^{n}\setminus\{0\}$ (\autocite{Spe1969, Smi1970, Kal1994}). The notion of $\R$-ellipticity has been characterised by means of Fourier multipliers \autocite{Mi1956} and singular integrals \autocite{CZ1956} that the Korn-type Inequality 
$$\forall p\in (1,\infty)\exists C>0\forall \zeta \in \CON ^{\infty}_{\com}(\R ^{n};V)\colon \ \norm{\nabla \zeta}_{\LEB^{p}}\leq C \norm{\A \zeta}_{\LEB^{p}}$$
is satisfied by the differential operator $\A$. In the super-linear growth case, Conti and Gmeineder showed that this allows to reduce the question of partial Hölder regularity of a local minimiser of a functional of the form $$\mathcal{F}[u;\Omega]=\int_{\Omega}f\left(\A u\right) \d x,$$
where $f\in\CON (W)$ is of $p$-growth ($p>1$), to the full gradient case \autocite{CG2020}.  Ornstein's \textit{Non-Inequality} \autocite[Theorem 1]{Orn1962}, \autocite{CFM2005,KK2011}, stating that there is no non-trivial Korn Inequality in the $\LEB ^{1}$-setting, implies that such a reduction is impossible in the linear growth regime.  Consequently, it had constituted a highly non-trivial task to adapt the full gradient case \autocite{GK2019} to the symmetric gradient case \autocite{Gme2020} in the linear growth regime.

\subsection{Partial Hölder Regularity in the Linear Growth Regime} 
In the convex case, partial regularity for linear growth functionals has been known in the convex context following the work of \citeauthor{AG1988} \autocite{AG1988} (also see \autocite{Sch2014, Gme2020-dirichlet} for variations of this theme). However, the methods employed therein are confined to the convex case. In the linear growth context, the key difficulty to overcome is the lack of weak compactness. This concerns the existence of minimisers as much as their regularity theory. In particular, this excludes indirect methods like the by now classical $\mathcal{B}$-\textit{Harmonic Approximation Lemma} \autocite{DG2000, DM2009, DGK2005}. A direct approach was needed to construct a $\mathcal{B}$-harmonic approximation. \citeauthor{GK2019} solved this problem by showing that the traces of $\BV$-maps on spheres of radius $R$ enjoy for $\mathcal{L}^{1}$-almost all sufficiently small radii $R$ more regularity than the default $\LEB^{1}$-regularity. This is called a \textit{Fubini-type property} of $\BV$-maps and  it was in effect the key point to construct a $\mathcal{B}$-harmonic approximation by solving the elliptic system

\begin{equation*}
\ccases{
    \begin{aligned}
    -\Div\left(\mathcal{B} \nabla h\right) &= 0 &\quad& \text{in}\  B_{R}(x_{0})\\
     h &= u &\quad& \text{on}\ \partial B_{R}(x_{0})
    \end{aligned}
    }.
\end{equation*}
We note that the solution operator  
$$(\CON^{0} \cap \LEB ^{1})\left(\partial B_{R}(x_{0});\R ^{N}\right)\ni u\mapsto h \in \SOB^{1,1}\left(B_{R}(x_{0});\R ^{N}\right)$$
associated to this system cannot be bounded as an operator from $\LEB ^{1}$ to $\SOB ^{1,1}$. In other words, without more regularity of $\Tr_{B_{R}(x_{0})}(u)$ we lack tools to precisely measure how close the $\mathcal{B}$-harmonic map $h$ is to $u$. This is why the \textit{Fubini-type property} of $\BV$-maps is essential for a direct approach in the linear growth regime.

\citeauthor{Gme2020} was able to adapt the ideas used in \autocite{GK2019} to the scenario where the full gradient is replaced by the symmetric gradient \autocite{Gme2020}. The main difficulties were to prove a Fubini-type property and, building on the latter, to prove precise estimates  for $u-h$, where $h$ denotes a suitable $\mathcal{B}$-\textit{harmonic approximation} of $u$.

\subsection{The Main Theorem}

Our scope is to show that the method for the symmetric gradient case extends to an entire class of first order homogeneous differential operators with constant coefficients, namely the class of $\C$-elliptic operators.

In line with the full \autocite{GK2019} and symmetric gradient case \autocite{Gme2020}, we will work from now on under the following assumptions:

\begin{itemize}
    \item[\namedlabel{itm:H0}{(H0)}] The differential operator $\A$ is $\C$-elliptic.
    \item[\namedlabel{itm:H1}{(H1)}] The integrand $f\in \CON^{2,1}_{\loc}(W)$ is of linear growth.
    \item[\namedlabel{itm:H2}{(H2)}] The integrand $f$ is strongly $V_{1}$-$\A$-quasi-convex, where $V_1$ denotes the reference integrand to be defined in the the upcoming section on preliminaries: There exists $\nu >0$ such that $F=f-\nu V_{1}\circ |\cdot |$ is $\A$-quasi-convex, \ie, for all $w\in W$ and all $\zeta\in \CON^{\infty}_{\com}(\mathbb{B})$ we have
    $$F(w)\leq \fint_{\mathbb{B}}F\left(w+\A \zeta\right)\d x. $$
\end{itemize}

For $\omega\subset \R^{n}$ open and bounded we associate to the integrand $f$ the functional $\mathcal{F}[u;\omega]=\int_{\omega}F(\A u (x))\d x$, where $u\in \SOB^{\A,1}(\omega)$.

We recall that to any $\R$-\textit{elliptic potential} $\A$ exists an \textit{annihilator}, see \autocite{VS2013}, $\mathcal{A}=\sum_{|\alpha|=k}\mathcal{A}_{\alpha}\partial_{\alpha}$ with $\mathcal{A}_{\alpha} \in \mathcal{L}(W;V) $ such that the \textit{symbol complex}
\[
V \xrightarrow{\A[\xi]} W \xrightarrow{\mathcal{A}[\xi]} V
\]
is exact for all $\xi \in \R ^{n}\setminus \{0\}$. This shows that the notion of $\A$-quasi-convexity is equivalent to the more widely known notion of $\mathcal{A}$-quasi-convexity which is strongly linked to weak sequential lower semi-continuity of the functional $\mathcal{F}$, which Fonseca and Müller showed in \autocite{FM1999}.

We fix $\Omega\subset \R ^{n}$ an open and bounded Lipschitz domain and $g\in \SOB^{\A,1}(\Omega)$ a prescribed boundary datum. Due to the lack of weak compactness, analogously to the full gradient case, the task to minimise $\mathcal{F}$ within a given Dirichlet class $g+\SOB^{\A,1}_{0}(\Omega)$ cannot be tackled by plainly applying the \textit{direct method}. Hence, we pass from the Sobolev-type space $\SOB^{\A,1}$ to the $\BV$-type space $\BVA$, hoping for better compactness with respect to the \textit{weak}$^{*}$-topology on the latter space. Already in the full gradient case, this forces us to somehow relax our functional $\mathcal{F}$ to this larger space $\BVA$. Without the assumption of $\C$-ellipticity, this relaxation procedure cannot be implemented analogously:  We recall that in the full gradient case, Alberti's rank-one result \autocite{Al1993} for $\BV$-maps has proven to be essential in order to obtain an integral representation of the Lebesgue-Serrin extension \autocite{AD1992}. Using that quasi-convexity implies rank-one convexity, Alberti's result ensures that the strong recession function of a quasi-convex integrand with linear growth is well-defined on the rank-one cone. The requisite result paralleling Alberti's result in the $\R$-elliptic case has been established by \citeauthor{PR2016} in \autocite{PR2016}. \textit{Weak}$^{*}$-compactness of closed, norm bounded sets in $\BVA$ as well as the existence of a strictly continuous and linear trace operator $\Tr _{\Omega}\colon \BVA (\Omega)\to \LEB ^{1}(\partial \Omega)$ for open and bounded Lipschitz domains are two features \textit{exclusive} to the $\C$-elliptic case \autocite[Theorem 1.1]{BDG2017}, \autocite[Theorem 1.1]{GR2019}. Since  the trace-operator on $\BVA$ is discontinuous with respect to weak$^{*}$-convergence, the boundary condition is no longer reflected by the space $\BVA$ but rather by the relaxed functional itself. The boundary condition then is encoded by a so called penalty term, which already pops up in the full gradient case \autocite{KR2010}:
$$\textnormal{P}_{f,\Omega,g}[u]=\int_{\partial \Omega}f^{\infty}\big( \nu _{\partial \Omega}\otimes_{\A}\Tr_{ \Omega}(u-g)\big)\d \mathcal{H}^{n-1}, $$
where $f^{\infty}(w)=\limsup_{t\to\infty}\frac{f(tw)}{t}$ denotes the strong recession function and $\xi\otimes_{\A}v=\sum_{\alpha=1}^{n}\xi _{\alpha}\A _{\alpha}v$ for $\xi \in \R^{n}$ and $v\in V$. We stress that it is necessary to assume that $\A$ is $\C$-elliptic in order for this integral expression to be well-defined. The relaxed functional then takes the form (\autocite[Section 5]{BDG2017}, \autocite{ADR2020})
$$\overline{\mathcal{F}}_{g}[u;\Omega]=\int_{\Omega}f(\A u) +\textnormal{P}_{f,\Omega,g}[u]. $$

We note that \ref{itm:H2} implies that there exist $b\in \R$ and $c>0$ such that for all $\zeta\in g+\SOB^{\A,1}_{0}(\Omega)$ we have $\mathcal{F}[\zeta;\Omega]\geq c V_{1}(\fint_{\Omega}|\A \zeta|\d x)+b$. This can be inferred from an extension and gluing argument similar to the argument in \autocite[218--219]{CK2017}. Identifying $\overline{\mathcal{F}}_{g}[u;\Omega]$ as the Lebesgue-Serrin extension \autocite[Section 5]{BDG2017}
$$\overline{\mathcal{F}}_{g}[u;\Omega]=\inf\Big\{\liminf_{j\to\infty}\mathcal{F}[u_{j};\Omega]\colon\ (u_{j})\subset g+\SOB^{\A,1}_{0}(\Omega),\ u_{j}\stackrel{\BVA}{\rightharpoonup^{*}} u \Big\}$$
yields coercivity of the relaxation. Hence, under our assumption, generalised minimisers  subject to a given Dirichlet boundary condition exist by means of the direct method. Since we are only striving for a \textit{local} regularity result, it is natural to consider the class of \textit{local generalised minimisers}:
\begin{definition*}
We call a $\BVA _{\loc}(\Omega)$-map $u$ local generalised minimiser of $\mathcal{F}$ if for any $\omega \Subset \Omega$ open and bounded Lipschitz domain and all $\zeta \in \BVA (\omega)$ we have
$$\overline{\mathcal{F}}_{u}[u;\omega]\leq  \overline{\mathcal{F}}_{u}[\zeta;\omega].$$
\end{definition*}

At this stage, we are ready to formulate the main theorem:

\begin{theorem}
Let us assume that \ref{itm:H0}, \ref{itm:H1}, and \ref{itm:H2} hold. Furthermore, let $u\in \BVA _{\loc}(\Omega)$ be a local generalised  minimiser of the to $f$ associated functional $\mathcal{F}$. Let $\alpha \in (0,1)$, let $M>0$ and let $B=B_{r}(x_{0})\Subset \Omega$ be a ball. Then there exists $\varepsilon>0$ depending on $M,\alpha, F,n,d_{V}, d_{W}$ and $\frac{L}{\nu}$ such that whenever we have
$$ |(\A u)_{B}|\leq M,\ \textnormal{and}\ \fint_{B}V_{1}\left (\A u-(\A u)_{B}\right)<\varepsilon,$$
then $u\restr{B}$ belongs to the class $\CON ^{1,\alpha}(B;V)$. In particular, the singular set $\Sigma_{u}$ defined by
$$\left\{x\in \Omega\colon \limsup_{r\to 0}|(\A u)_{B_{r}(x)}|=+\infty\right\}\bigcup\left\{x\in \Omega\colon \liminf_{r\to 0}\fint_{B}V_{1}\left (\A u-(\A u)_{B}\right)>0\right\}$$
is a relatively closed Lebesgue-null-set and we have $u\in \CON^{1,\alpha}_{\loc}(\Omega\setminus \Sigma_{u};V)$  for all $\alpha \in (0,1)$.

\end{theorem}

We wish to point out that for the present paper, the assumption of $\C$-ellipticity is crucial and visible on several stages (so \eg\ in the very definition of the functionals where boundary traces come into play); the elliptic case, however, seems to require refined methods.

%% file: draft-final/preliminaries.tex
\section{Preliminaries}

\subsection{Notation}
For a finite dimensional real vector space $Z$ we use the shorthand notation $d_{Z}=\dim_{\R}Z$. Furthermore, we will suppress the target vector space when dealing with different function spaces. For example, if $u\colon \Omega \to Z$ is a $\LEB ^{1}$-map, we simply write $u\in \LEB^{1}(\Omega)$ instead of $u\in \LEB^{1}(\Omega;Z)$. Within the context, it will be clear which target vector space we are referring to. Furthermore, by $|\cdot |$ we will denote any norm of a finite dimensional, normed vector space such as $\mathcal{L}(V;W)$, $V$ or $\mathcal{L}(V\times V;W)$. Since all norms of a finite dimensional normed vector space are equivalent, this is an non-problematic convention. Throughout, we fix an orthonormal basis $(v_{j},...,v_{N})$ of $V$, \ie, $N=d_{V}$. Furthermore, $$\left(e_{jk}\right)_{j=1,k=1}^{j=N,k=n}=\delta_{jk}$$ denotes the standard basis of $\R^{N\times n}$.

Integration with respect to the $(n-1)$-dimensional Hausdorff-measure $\mathcal{H}^{n-1}$ will be denoted by $\d \sigma _{x}$, where $x$ is integration variable.

As usual, $B_{r}(x_{0})\subset \R ^{n}$ denotes the open ball with centre $x_{0}$ and radius $r>0$. Often we will suppress the centre of the ball if it is clear within the context and we will simply write $B_{r}$. Furthermore, we put $\mathbb{B}=B_{1}(0)$ and $\mathbb{S}=\partial\mathbb{B}$.

By $\mathcal{M}(\Omega;Z)$ we denote the space of $Z$-valued finite Radon measures on $\Omega$. For $\mu \in \mathcal{M}(\Omega)$ and open, bounded subsets $\omega \subset \Omega$, the total variation-measure of $\mu$ will be denoted by $|\mu|$ and the average of $\mu$ on $\omega$ with respect to the Lebesgue measure will be written as
$(\mu)_{\omega}\coloneq\frac{\mu(\omega)}{\mathcal{L}^{n}(\omega)}.$

Let $V_{1}(t)=\sqrt{1+t^{2}}-1$ denote the \textit{reference integrand}. We will use the shorthand notation $V_{1}(z)=V_{1}(|z|)$.

For a $\CON ^{1}$-function $G\colon Z\to \R $ and $z_{0}\in Z$ we denote by
$$(G)_{z_{0}}\colon Z\to \R,\ z\mapsto G(z_{0}+z)-\Big(G(z_{0})+DG(z_{0})[z]\Big)$$
the \textit{linearisation} of $G$ at $z_{0}$.

For $k\in\N$ we denote by $\mathcal{P}_{k}(Z)$ the vector space of all polynomials $$p\in \R [X_{1},...,X_{n}]\otimes_{\R}Z$$ of degree at most $k$.

By $C$ we will denote a generic constant which may vary from line to line. Since it is very important throughout on which parameters a constant depends on, we will write for example $C(M,p)$ if the constant depends on $M$ and $p$. 

For $\xi \in \R^{n}$ and $v\in V$ we put $\xi\otimes_{\A}v=\sum_{\alpha=1}^{n}\xi _{\alpha}\A _{\alpha}v$. Furthermore, we denote by
$$\mathcal{R}(\A)\coloneq\spann\{\xi\otimes_{\A} v\colon \xi \in \R^{n},v\in V\} $$
the \textit{effective range} of $\A$ and we call
$$\mathcal{N}(\A)\coloneq\{u\in \mathcal{D}^{*}\colon \A u=0\}$$ the \textit{null-space} of $\A$.
The \textit{formally adjoint} operator $\A^{*}$ is here defined by the formula
$$\A^{*}=-\sum\limits_{\alpha =1}^{n}\A _{\alpha}^{*}\partial _{\alpha}. $$

Let $\Omega \subset \R ^{n}$ be open and let $\mathbb{M}\subset \R^{n}$ be an embedded $(n-1)$-dimensional $\CON ^{1}$-submanifold of $\R ^{n}$. For $\alpha\in(0,1)$ and $p\in[1,\infty)$, we recall the definition of the fractional Sobolev space (semi)-norms on $\Omega$ and $\mathbb{M}$, respectively:
\begin{itemize}
    \item $[u]_{\SOB^{\alpha,p}(\Omega)}=\Big(\int_{\Omega^{2}}\frac{|u(x)-u(y)|^{p}}{|x-y|^{n+\alpha p}}\d x\d y\Big)^{\frac{1}{p}}$, $\norm{u}_{\SOB^{\alpha,p}(\Omega)}=\norm{u}_{\LEB ^{p}(\Omega)}+[u]_{\SOB^{\alpha,p}(\Omega)}$,
    \item $[u]_{\SOB^{\alpha,p}(\mathbb{M})}=\Big(\int_{\mathbb{M}^{2}}\frac{|u(x)-u(y)|^{p}}{|x-y|^{n-1+\alpha p}}\d \sigma_{x}\d\sigma_{y}\Big)^{\frac{1}{p}}$, $\norm{u}_{\SOB^{\alpha,p}(\mathbb{M})}=\norm{u}_{\LEB ^{p}(\mathbb{M})}+[u]_{\SOB^{\alpha,p}(\mathbb{M})}$.
\end{itemize}

\subsection{Space of maps of bounded $\A$-variation}

We are going to collect prerequisites on $\A$-weakly differentiable maps.
In the spirit of \autocite{BDG2017}, we define  Sobolev- and $\BV$-type spaces as follows:
\begin{definition}
Let $\Omega\subset \R ^{n}$ be open and let $p\in [1,\infty]$. We define:

\begin{itemize}
    \item $\SOB^{\A,p}(\Omega)\coloneqq\{u\in \LEB^{p}(\Omega;V)\colon\ \A u\in \LEB^{p}(\Omega;W)\},$ and
    \item $\BVA (\Omega)\coloneqq\{u\in \LEB^{p}(\Omega;V)\colon\ \A u\in \mathcal{M}(\Omega;W)\}.$
\end{itemize}
These spaces can be equipped with the obvious norms making them Banach spaces. Also the spaces  $\SOB^{\A,p}_{0}(\Omega)$ are as usual defined as the closure of $\CON ^{\infty}_{\com}(\Omega;V)$ with respect to the according norm.
\end{definition}
Let $u\in \BVA_{\loc}(\Omega)$. Then we consider the Radon-Nikodým decomposition of $\A u$ with respect to the Lebesgue measure $\A u= \A^a u + \A^s u$, where $\A^a u$ denotes the absolutely continuous part and $\A^s u$ the singular part. Next, we recall different notions of convergence in $\BVA$:

\begin{definition}
Let $u \in \BVA(\Omega)$ and $(u_j) \subset \BVA(\Omega)$. Then $u_j$ converges to $u$ in the
\begin{enumerate}[label=(\roman*)]
    \item \emph{$\A$-weak*-sense} ($u_j \overset{\ast}{\rightharpoonup} u$) if $u_j \to u$ strongly in $L^1(\Omega)$ and $\A u_j \overset{\ast}{\rightharpoonup} \A u$ in the weak*-sense of $W$-valued Radon measures on $\Omega$.
    \item \emph{$\A$-strict sense} ($u_j \overset{s}{\to} u$) if $u_j \to u$ strongly in $L^1(\Omega)$ and $\lvert \A u_j \rvert(\Omega) \to \lvert \A u \rvert(\Omega)$.
    \item \emph{$\A$-area-strict sense} ($u_j \overset{\langle \cdot \rangle}{\to} u$) if $u_j \to u$ strongly in $L^1(\Omega)$ and
    $$
        \int_{\Omega} \sqrt{1 + \left| \frac{\d \A^a u _j}{\d \mathcal{L}^{n}} \right|^2} \d x + \lvert \A^s u_j \rvert(\Omega) \to \int_{\Omega} \sqrt{1 + \left | \frac{\d \A^a u }{\d \mathcal{L}^{n}} \right|^2} \d x + \lvert \A^s u \rvert(\Omega)
    $$
\end{enumerate}
\end{definition}
\begin{lemma}{\autocite[Theorem 2.8, Lemma 4.15]{BDG2017}}
Let $\Omega\subset \R ^{n}$ open. Then $(\CON ^{\infty}\cap\BVA )(\Omega)$ is dense in $\BVA(\Omega)$ with respect to the strict and area-strict topologies. If $\Omega$ is additionally a bounded Lipschitz domain, then $\CON ^{\infty}(\overline{\Omega})$ is dense in $\BVA (\Omega)$ with respect to the strict and area-strict topologies. Let $u_0 \in W^{\A,1}(\Omega)$.  For each $u \in \BVA(\Omega)$  there exists a sequence $(u_j) \subset u_0 + C_c^{\infty}(\Omega)$ such that $\lVert u_j - u \rVert_{L^1(\Omega)} \to 0$ and
\begin{align*}
    \int_{\Omega} \sqrt{1 + \left| \frac{\d \A^a u _j}{\d \mathcal{L}^{n}} \right|^2} &\to \int_{\Omega} 
\sqrt{1 + \left| \frac{\d \A^a u }{\d \mathcal{L}^{n}} \right|^2} \d x + \lvert \A^s u \rvert(\Omega)\\
    &\phantom{\to}+ \int_{\partial \Omega} \lvert (\Tr(u) - \Tr(u_0)) 
\otimes_{\A} \nu_{\partial \Omega} \rvert \d \mathcal{H}^{n-1} \qquad \text{for $j \to \infty$.}
\end{align*}
\end{lemma}
The proof of the last assertion is analogous to the $\BD$-case \autocite{Gme2020}.
\begin{lemma}{\autocite[Theorem 3.2]{BDG2017}}
\label{lem: Poincaré}
Let $B\subset \R ^{n}$ be an open ball of radius $r>0$ and let $\Pi_{B}$ denote the $\LEB ^{2}(B)$-projection onto $\mathcal{N}(\A)$. Then there exists a constant $C>0$ such that for all $u\in \BVA (B)$ we have
$$\norm{u-\Pi_{B}u}_{\LEB^{1}(B)}\leq Cr|\A u|(B).$$
\end{lemma}

\begin{lemma}{\autocite[Theorem 1.2]{BDG2017}}
Let $\Omega\subset \R ^{n}$ be an open and bounded Lipschitz domain. Then there exists a linear and strictly continuous operator $\Tr_{\Omega}\colon \BVA (\Omega)\to \LEB^{1}(\partial \Omega)$ such that for all $u\in \CON ^{1}(\overline{\Omega})$ we have $\Tr_{ \Omega}u=u\restr{\partial \Omega}.$
For an open and bounded Lipschitz subset $\Omega^{\prime} \Subset \Omega$, we consider so called \emph{interior} and \emph{exterior} traces of $u$ denoted by
$$\Tr^{-}_{ \Omega^{\prime}}(u)\coloneq \Tr_{ \Omega^{\prime}}\left(u\restr{\Omega^{\prime}}\right)\ \textnormal{and}\  \Tr^{+}_{ \Omega^{\prime}}(u) \coloneq \Tr_{\Omega \setminus \Omega^{\prime}}\left(u\restr{\Omega \setminus \Omega^{\prime}}\right).$$
One can explicitly compute 
\begin{equation}
\label{eq:outer-trace-approx-limit}
	\lim_{r \searrow 0} \slashint_{B^{\pm}(x,r)} \lvert u(y) - \Tr^{\pm}_{\partial B}(u)(x) \rvert \d y = 0
\end{equation}
for $\mathcal{H}^{n-1}$-a.e. $x \in \partial B$, where $B^{\pm}(x,r) \coloneq \lcbrace y \in B_r(x) \mid \langle y - x, \nu(x) \rangle \gtrless 0 \rcbrace$. Here, $\nu(x)$ designates the outer unit normal vector to the sphere $\partial B$ at point $x$. 

\end{lemma}
\begin{proposition}{\autocite[Proposition 5.1]{BDG2017}}
\label{prop: continuity of functional}
Let $\Omega\subset \R^{n}$ be open, bounded and let $g\colon W \to \R$ be an $\A$-quasi-convex integrand of linear growth. Then the functional
$$\overline{\mathcal{G}}\colon \BVA (\Omega)\to \R, u\mapsto \int_{\Omega}g\left(\A u \right)\coloneq\int g\left(\frac{\A ^{a} u}{\mathcal{L}^{n}}\right)\d x+\int_{\Omega}g^{\infty}\left(\frac{\d \A ^{s} u}{\d |\A ^{s}u|}\right)\d |\A ^{s}u|$$ is $\A$-area strictly continuous and sequentially lower semi-continuous with respect to weak$^{*}$-convergence.
\end{proposition}
We will use the shorthand notation $\fint_{\Omega}g(\A u)=\frac{\int_{\Omega}g(\A u)}{\mathcal{L}^{n}(\Omega)}$.

\begin{lemma}
\label{lem:bva-slobodeckij-embedding}
Let $n \geq 2$, $\alpha \in (0,1)$, let $B_{2r} \subset \R^n$ be a ball of radius $2r > 0$ and let $p \coloneq \frac{n}{n-1+\alpha}$. Then there exists a constant $C>0$ independent of the radius $r$  such that for every ball $B_r \subset \R^n$ and every $u \in \BVA(\R^n)$, there exists some $b \in \mathcal{N}(\A)$ with
\begin{equation}
\label{eq:slobodeckij-seminorm-estimate}
	\PARENS{ \slashint_{B_r} \int_{B_r} \frac{\lvert u_b(x) - u_b(y) \rvert^p}{\lvert x - y \rvert^{n + \alpha p}} \d x \d y }^{\frac{1}{p}} \leq C r^{1-\alpha} \slashint_{B_{2r}} \lvert \A u \rvert,
\end{equation}
where $u_b \coloneq u - b$.
\end{lemma}

\begin{proof}
Let $\wtilde{u} \coloneq u \varphi$, where $\varphi \in C_c^{\infty}(B_{2r};[0,1])$ is a bump function with $\mbf{1}_{B_r} \leq \varphi \leq \mbf{1}_{B_{2r}}$ and $\lvert \nabla \varphi \rvert \leq C/r$. Then $\wtilde{u} \in \BVA(\R^n)$ and 
\begin{equation}
\label{eq:slobodeckij-auxiliary-extension-estimate}
    \lVert \A \wtilde{u} \rVert_{L^1(B_{2r})} \leq C(r) \lVert u \rVert_{\BVA(B_{2r})}.
\end{equation}
Noting that $\Tr(\wtilde{u}) = 0$ on $\partial B_{2r}$, there exists a sequence $(\wtilde{u}_j) \subset C_c^{\infty}(B_{2r};\R^N)$ such that $\wtilde{u}_j \to \wtilde{u}$ strictly. Since $\A$  is $\C$-elliptic it is in particular $\R$-elliptic and canceling, see \autocite{GR2019}. Applying \autocite[Proposition 8.11]{VS2013} we obtain
\begin{equation*}
    \lVert \wtilde{u}_j \rVert_{W^{\alpha, p}(B_{r})}\leq \lVert \wtilde{u}_j \rVert_{W^{\alpha, p}(\R^n)} \leq C \lVert \A \wtilde{u}_j \rVert_{L^1(\R^n)}
\end{equation*}
By passing to a subsequence we may assume that $ (\wtilde{u}_j)$ converges $\mathcal{L}^{n}$-almost everywhere. By Fatou's Lemma and the strict convergence we obtain
\begin{equation}
    \label{auxiliary estimate: Sobolev_slobodeckji}
    \norm{u}_{W^{\alpha, p}(B_{r})}\leq C(r)\norm{u}_{\BVA(B_{2r})} .
\end{equation}
We put $b\coloneq\Pi_{B_{r}}u$, $(u-b)_r \coloneq (u-b)(rx)$ for $x\in \mathbb{B}$ and eventually we obtain by Poincaré's Inequality \ref{lem: Poincaré}, scaling and applying \ref{auxiliary estimate: Sobolev_slobodeckji} for $r=1$:
\begin{equation*}
        \PARENS{ \fint_{B_r} \int_{B_r} \frac{\lvert u_b(x) - u_b(y) \rvert^p}{\lvert x - y \rvert^{n + \alpha p}} \d x \d y }^{\frac{1}{p}}\leq Cr^{-\alpha}\norm{(u-b)_{r}}_{W^{\alpha, p}(B_1)}\leq Cr^{1-\alpha}\fint_{B_{2r}}|\A u|.
\end{equation*}
\end{proof}

\begin{lemma}
\label{lem: mean value}
Let $\B\Subset \Omega$ be a ball and $u\in\BVA_{\loc}(\Omega)$. Then there exists a linear map $a\colon \R ^{n}\to V$ such that $\A a= (\A u)_{B}$.
\end{lemma}
\begin{proof}
Let $u_{j}\in \CON ^{\infty}(\overline{B};V)$ such that $u_j \overset{s}{\to} u$ as $j\to \infty$. Clearly, we have $\sup_{j}|(\A u_{j})_{B}|<\infty$ and $(\A u_{j})_{B}\in \mathcal{R}(\A )$, since
$\A u_{j}(B)=\sum_{\alpha=1}^{n}\A _{\alpha} \fint_{B} \partial_{\alpha}u_{j} \d x$. After extraction of a non-relabelled sub-sequence we find  $w\in \mathcal{R}(\A )$ such that $(\A u_{j})_{B}\to w$ in $W$ as $j\to \infty$. We find $v_{\alpha }\in V $ such that $w=\sum_{\alpha=1}^{n}\A _{\alpha}v_{\alpha}$ and  put $a[x]=\sum_{\alpha=1}^{n}x _{\alpha}v_{\alpha}$. This yields the claim.
\end{proof}
\subsection{Linearisation and the Reference Integrand}
\begin{lemma}
\label{lem:linearisation and refernece integrand}
Let $f\colon W\to \R$ satisfy \ref{itm:H1} and \ref{itm:H2}, let $m>0$. Then there exists a constant $C(m)>0$ such that for all $w_{0}\in W$ with $|w_{0}|\leq m$, all $\xi \in \R ^{n}$ and all $v\in V$ we have
\begin{equation}
\label{auxiliary estimate linearisation}
\begin{split}
    D^{2}f(w_{0})[\xi\otimes_{\A}v,\xi\otimes_{\A}v]\geq \frac{C(m)}{\nu}|\xi\otimes_{\A}v|^{2},\\
    |D^{2}(f)_{w_{0}}[w,\cdot]-D(f)_{w_{0}}(w)|\leq C(m)V_{1}(w).
\end{split}
\end{equation}
Also, it is worth noting that we have
\begin{equation}
    \label{refernece integrand quadratic growth}
    0<\inf\limits_{t\in (0,1)}\frac{V_{1}(t)}{t^{2}}\leq \sup\limits_{t\in (0,1)}\frac{V_{1}(t)}{t^{2}}<\infty,
\end{equation}
\begin{equation}
    V_{1}(rt)\leq rV_{1}(t)\ \textnormal{for}\ r\in (0,1),\quad V_{1}(rt)\leq r^{2}V_{1}(t)\ \textnormal{for}\ r\in (1,\infty), 
\end{equation}
\begin{equation}
    V_{1}(|s|+|t|)\leq 4(V_{1}(|s|)+|V_{1}(t|)).
\end{equation}
\end{lemma}
The proof is analogous to the proof of \autocite[Lemma 4.2]{GK2019} and  \autocite[Lemma 5.1]{Gme2020}. 

\subsection{Caccioppoli Inequality of the second kind}
The Caccioppoli Inequality is indispensable for our proof of partial regularity. However, it is line for line analogous to the full and symmetric the gradient case \autocite[Proposition 4.3]{GK2019}, \autocite[Proposition 5.2]{Gme2020}, replacing  $\nabla$ or $\varepsilon$ with $\A$, respectively: by exploiting the strong $V_{1}$-$\A$-quasi-convexity and minimality of $u$, we apply Widman's hole-filling trick and iterate the resulting inequality.
\begin{proposition}
\label{prop:caccioppoli}
We assume that \ref{itm:H0}, \ref{itm:H1}, and \ref{itm:H2} hold. Let $u \in \BVA_{\loc}(\Omega)$ be a local generalised minimiser of the functional $\mathcal{F}$ and let $a\colon \R^n \to \R^N$ be an affine map with $\vert \A a \vert \leq m$ for some $m > 0$. Then there exists a constant $C = C(d_{V};d_{W},\frac{L}{\nu}) \in (1, \infty)$ 
such that
\[
\int_{B_{r/2}(x_0)} V(\A (u-a)) \leq C \int_{B_r(x_0)}V_1\PARENS{\frac{u-a}{r}}\d x
\]
for any ball $B_r(x_0) \Subset\Omega$.
\end{proposition}


\subsection{The Ekeland Variational Principle}
\begin{lemma}{\autocite[Theorem 1.1]{Ek1974}}
\label{lem:ekeland}
Let $(X,d)$ be a complete metric space and let $\mathcal{G}: X \to \R \cup \{+\infty\}$ be a lower semicontinuous function for the metric topology, bounded from below and taking a finite value at some point. Assume that for some $x \in X$ and some $\epsilon > 0$ we have
\[
    \mathcal{G}(u) \leq \inf_X \mathcal{G} + \epsilon.
\]
Then, there exists $\tilde{x} \in X$ such that
\begin{enumerate}[label=(\roman*)]
    \item $d(x,\tilde{x}) \leq \SQRT{\epsilon}$,
    \item $\mathcal{G}(\tilde{x}) \leq \mathcal{G}(x)$,
    \item $\mathcal{G}(\tilde{x}) \leq \mathcal{G}(y) + \SQRT{\epsilon} d(\tilde{x},y)$ for all $y \in X$.
\end{enumerate}
\end{lemma}


\subsection{Estimates for Elliptic systems}

\begin{lemma}
\label{lem:estimates elliptic stystems}
We are going to consider a strongly $\A$-Legendre-Hadamard elliptic bilinear form $\mathcal{B}\colon (\mathcal{R}(\A ))^{2}\to \R$, \ie, there exist $\alpha, \beta >0$ such that for all $\xi\in \R ^{n}$ and $v\in V$ we have
$$\mathcal{B}[\xi \otimes_{\A} v, \xi \otimes_{\A} v] \geq \alpha |\xi \otimes_{\A} v|^{2}, \ \textnormal{and}\ |\mathcal{B}|\leq \beta .$$
\begin{itemize}
    \item[(i)] For every $g\in \SOB^{\frac{1}{n+1},\frac{n+1}{n}}(\mathbb{S};V)$ there exists a unique weak solution $h\in \SOB^{1,\frac{n+1}{n}}(\mathbb{B};V)$ of the elliptic system:
    \begin{equation*}
    \ccases{
    \begin{aligned}
    \A ^{*}(\mathcal{B} \A h) &= 0 &\quad& \text{in}\  \mathbb{B}\\
     h &= g &\quad& \text{on}\ \mathbb{S}.
    \end{aligned}
    }
    \end{equation*}
    Furthermore, there exists a positive constant $C=C(d_{V},d_{W}, n, \frac{\beta}{\alpha})$ such that we have the estimates 
    \begin{equation}
    \label{estimate elliptic system homogeneous with boundary datum}
        \norm{h}_{\SOB ^{1,\frac{n+1}{n}}(\mathbb{B})} \leq C\norm{g}_{\SOB^{\frac{1}{n+1},\frac{n+1}{n}}(\mathbb{S})}\ \textnormal{and}\ \norm{\nabla h}_{\LEB^{\frac{n+1}{n}}(\mathbb{B})}\leq C [g]_{\SOB^{\frac{1}{n+1},\frac{n+1}{n}}(\mathbb{S})}.
    \end{equation}
    
    \item[(ii)] For every $g\in \LEB^{\infty}(\mathbb{B};W)$ and every $p>n$ there exists a unique solution $u\in (\SOB^{1,\infty}\cap \SOB^{1,p}_{0})(\mathbb{B};V)$ of the elliptic system:
    \begin{equation}
    \label{eq:linearized-hadamard-system 2}
    \ccases{
    \begin{aligned}
    \A ^{*}(\mathcal{B} \A u) &= g &\quad& \text{in}\ \mathbb{B}\\
     u &= 0 &\quad& \text{on}\ \mathbb{S}.
    \end{aligned}
    }
    \end{equation}
    Furthermore, there exists a positive constant $C=C(p,d_{V},d_{W}, n, \frac{\beta}{\alpha})$ such that $\norm{u}_{\SOB ^{1,\infty}(\mathbb{B})}\leq C\norm{g}_{\LEB ^{p}(\mathbb{B})}$.
    \item[(iii)] Moreover, if $h \in W^{\A, 1}(\Omega;V)$ satisfies 
    \begin{equation*}
    \A ^{*}(\mathcal{B} \A u) = 0 \quad \text{in 
     $\mathcal{D}^{\prime}(\Omega; V)$},
     \end{equation*} 
     then $u \in C^{\infty}(B_{r};V)$ and
     \begin{equation}
     \label{eq:sup-estimates-hadamard}
      \sup_{B_{r/2}} \lvert \nabla u \rvert + r \sup_{B_{r/2}} \lvert \nabla^2 u 
\rvert \leq
    C \slashint_{B_r} \lvert \nabla u \rvert \d x
\end{equation}
for all balls $B_r \coloneq B_r(x_0) \Subset \Omega$, where $C = C(n, d_{V},d_{W},\frac{\beta}{\alpha}) > 0$ is a constant.
\end{itemize}
\begin{proof}
We define the bilinear form $\tilde{\mathcal{B}}\colon (\R ^{N\times n})^{2}\to \R$ by the relations
$$\tilde{\mathcal{B}}[e_{j_{1}k_{1}},e_{j_{2}k_{2}}]=\mathcal{B}[\A _{k_{1}}v_{j_{1}},\A _{k_{2}}v_{j_{2}}] \ \textnormal{for}\ j_{1},j_{2}=1,...,N,\ k_{1},k_{2}=1,...,n. $$
Note that by construction we have that $\tilde{\mathcal{B}}$ is strongly Legendre-Hadamard elliptic, \ie, we have for all $z\in \mathcal{C}(N,n)$:
$$\tilde{\mathcal{B}}[z,z]\geq \alpha c_{\A}|z|^{2},\ \textnormal{and}\ |\tilde{\mathcal{B}}|\leq \beta N\sup\limits_{\alpha=1}^{n}|\A_{\alpha}|^{2}  . $$
Applying \autocite[Proposition 2.11]{GK2019} in combination with Morrey's Inequality yields $(i)$ and $(ii)$. For the gradient estimate in $(i)$ we note that we have $\nabla h= \nabla (h-\fint _{\mathbb{S}}g\d \mathcal{H}^{n-1})$ and that we have $\norm{g-\fint _{\mathbb{S}}g\d \mathcal{H}^{n-1}}\leq C[g]$ for a constant independent of $g$. The third item may be derived by means of the difference quotient method as it has been carried out on the level of first order derivatives in the proof of \autocite[Proposition 2.10]{CFM1998}.

\end{proof}
\begin{corollary} 
\label{corollary: elliptic estimates}
Let $B=B_{R}(x_{0})$, $r\in(\frac{38R}{40},\frac{39R}{40})$, $\tilde{B}=B_{r}(x_{0})$ and let $g\in\SOB^{\frac{1}{n+1},\frac{n+1}{n}}(\partial \tilde{B})$. We suppose that $h\in \SOB^{1,\frac{n+1}{n}}(\tilde{B};V)$ solves the elliptic system:
    \begin{equation}
    \ccases{
    \begin{aligned}
    \A ^{*}(\mathcal{B} \A u) &= 0 &\quad& \text{in $\tilde{B}$}\\
     u &= g &\quad& \text{on $\partial \tilde{B}$.}
    \end{aligned}
    }
    \end{equation}

Then there exists $C(n, d_{V},d_{W},\frac{\beta}{\alpha})>0$ such that for all $\sigma \in (0,\frac{1}{10})$ we have for $A_{h}[x]=h(x_{0})+\langle \nabla h(x_{0}) ,x-x_{0}\rangle$:
\begin{equation}
\label{auxiliary estimate harmonic map}
    \int_{B_{2\sigma R}}V_{1}\Big(\frac{h-A_{h}}{\sigma R}\Big)\d x \leq C \sigma^{n} R^{n}V_{1}\Big(\sigma r^{-\frac{n^{2}}{n+1}}[g]_{\SOB^{\frac{1}{n+1},\frac{n+1}{n}}(\partial \tilde{B})}\Big)
\end{equation}
Furthermore, we have:
\begin{equation}
    \sup\limits_{B_{\frac{r}{3}}}|\nabla h|\leq C r^{-\frac{n^{2}}{n+1}}[g]_{\SOB^{\frac{1}{n+1},\frac{n+1}{n}}(\partial \tilde{B})}.
\end{equation}

\end{corollary}
\begin{proof}
By Taylor's formula, we obtain in conjunction with \ref{eq:sup-estimates-hadamard} and Jensen's Inequality
\begin{equation*}
    \sup\limits_{B_{2\sigma R}(x_{0})}\Big |\frac{h-A_{h}}{\sigma R}\Big |\leq \sigma R \sup\limits_{B_{\frac{1}{5} R}(x_{0})}|\nabla^{2}h|\leq C \sigma\Big(\fint_{B_{\frac{2}{5} R}(x_{0})} |\nabla h|^{\frac{n+1}{n}}\d x\Big)^{\frac{n}{n+1}}.
\end{equation*}
Combining this with the estimate \ref{estimate elliptic system homogeneous with boundary datum}   yields the corollary.
\end{proof}

\end{lemma}

\subsection{Auxiliary Measure Theory}
\begin{lemma}
\label{lem:auxiliary-lebesgue-point-estimate}
Let $-\infty < a < b < \infty$ and let $J \subset (a,b)$ be a measurable subset with $\mathcal{L}^1((a,b) \setminus J) = 0$. Then for every $g \in L^1((a,b); \R_{\geq 0})$, there exists a Lebesgue point $\xi_0 \in J$ for $g$ such that
\[
    g^{\ast}(\xi_0) 
    = \lim_{r \searrow 0} \slashint_{\xi_0 - r}^{\xi_0 + r} g \d x
    \leq \frac{2}{b-a} \int_a^b g \d x,
\]
where $g^{\ast}$ is the precise representative of $g$.
\end{lemma}

The Lemma can be verified by means of a contraposition argument. One may assume without loss of generality that $\int_{a}^{b}g(x)\d x>0$. Then one integrates the reversed inequality to derive a contradiction to  the latter integral being positive.

%% file: draft-final/fubini-type-theorem.tex
\section{Fubini-type theorem}
\label{sec:fubini-type-theorem}
In this section, we are going to establish a Fubini-type property for $\BVA$-maps. Later on, this will prove essential in order to construct a $\mathcal{B}$-harmonic approximation of a given local generalised minimiser. We have seen that certain semi-norms of a fractional Sobolev space on some ball may be estimated from above by the total $\A$-variation on a larger ball. We will prove that for $\mathcal{L}^{1}$-almost every sufficiently small radii $R $, an $(n-1)$-dimensional analogous estimate holds on a sphere of radius~$R$.

\begin{theorem}
\label{thm:fubini-type}
Let $n \geq 2$ and $\alpha \in (0,1)$. Let further be $x_0 \in \R^n$, $R > 0$ and $u \in \BVA_{\loc}(\R^n)$. Then for $\mathcal{L}^1$-a.e. radius $r \in (0,R)$, the restrictions $u\restr{\partial B_r(x_0)}$ are well-defined and belong to the space $W^{\alpha, p}(\partial B_r(x_0); V)$, where $p \coloneq \frac{n}{n-1+\alpha}$.

Moreover, there exists a constant $C = C(\A, n, \alpha) > 0$, independent of $x_0$, $R$ and $u$, such that for all $0 < s < r < R$ there exists $t \in (s,r)$ with
\begin{equation}
\label{eq:fubini-type-estimate}
    \PARENS{ \slashint_{\partial B_t(x_0)} \int_{\partial B_t(x_0)} \frac{\lvert u_b(x) - u_b(y) \rvert^p}{\lvert x - y \rvert^{n-1+\alpha p}} \d \sigma_x \d \sigma_y }^{\frac{1}{p}}
    \leq C \frac{r^n}{t^{\frac{n-1}{p}} (r-s)^{\frac{1}{p}}} \slashint_{B_{2r}(x_0)} \lvert \A u \rvert
\end{equation}
for some suitable $b \in \mathcal{N}(\A)$.
\end{theorem}
\begin{remark*}
In the follwoing constellation $p=\frac{n+1}{n}$, $\alpha=\frac{1}{n+1}$ and $s>Cr$, the inequality \ref{eq:fubini-type-estimate} then takes the form
\begin{equation}
    \label{fubini type estimate lower dimensional}
    [u_{b}]_{\SOB^{\frac{1}{n+1},\frac{n+1}{n}}(\partial B_{r})}\leq Cr^{\frac{n^{2}}{n+1}}\fint_{B_{2r}}|\A u |.
\end{equation}
\end{remark*}
\begin{proof}
The proof is analogous to the one for \autocite[Theorem 4.1]{Gme2020} in the $\BD$-case and only requires minor modifications in the second and the third step.

Let $\theta \in (0,1)$, ${q \in [1,\infty)}$ and $u \in (W^{\theta, q} \cap C)(\R^n;V)$. Then it has been established \autocite[Theorem 4.1]{Gme2020} that there is a constant $C = C(n, \theta, q) > 0$ such that for all $R > 0$ we have
\begin{equation}
\label{eq:general-fubini-type-theorem}
    \int_0^R \iint_{\partial B_r \times \partial B_r} \frac{\lvert u(x) - u(y) \rvert^q}{\lvert x - y \rvert^{n-1+\theta q}} \d \sigma_x \d \sigma_y \d r
    \leq C \iint_{B_R \times B_R} \frac{\lvert u(x) - u(y) \rvert^q}{\lvert x - y \rvert^{n + \theta q}} \d x \d y.
\end{equation}

The aim now is to establish that $u$ may be explicitly evaluated $\mathcal{H}^{n-1}$-a.e. point wisely on $\mathcal{L}^1$-a.e. sphere centred at the origin. For that matter, let $u \in \BVA(\R^n)$ and let $0 < R_1 < R_2 < \infty$ be arbitrary. The set
\[
    I \coloneq \{ t \in (R_1, R_2) \mid \lvert \A u \rvert(\partial B_t) > 0 \}
\]
is at most countable and hence a $\mathcal{L}^{1}$-nullset. Now let $t \in (R_1, R_2) \setminus I$. Then by \autocite[Corollary 4.21]{BDG2017},
\[
    \A u \mres\partial B_t = (\Tr_{ B_t}^{+}(u) - \Tr_{ B_t}^{-}(u)) \otimes_{\A} \nu_{\partial B_t} \mathcal{H}^{n-1} \mres\partial B_t,
\]
where $\nu_{B_t}$ denotes the outer unit normal to $\partial B_t$. Hence, using that $t \in (R_1, R_2) \setminus I$ in the last step,
\[
    \int_{\partial B_t} \lvert (\Tr_{ B_t}^{+}(u) - \Tr_{ B_t}^{-}(u)) \otimes_{\A} \nu_{\partial B_t} \rvert \d \sigma
    = \lvert \A u \rvert(\partial B_t)
    = 0.
\]
As a result, we have $\lvert \Tr_{ B_t}^{+}(u) - \Tr_{ B_t}^{-}(u) \rvert = 0$ $\mathcal{H}^{n-1}$-a.e. on $\partial B_t$. Furthermore, writing $\wtilde{u}(x) \coloneq \Tr_{B_t}^{+}(u)(x) = \Tr_{ B_t}^{-}(u)(x)$ for such $x \in \partial B_t$, \eqref{eq:outer-trace-approx-limit} implies that
\begin{equation}
\label{eq:fubini-step2-auxiliary-lebesgue-limit}
    \lim_{r \searrow 0} \slashint_{B_r(x) \cap B_t} \lvert u - \wtilde{u}(x) \rvert \d y
    = \lim_{r \searrow 0} \slashint_{B_r(x) \cap \overbar{B_t}^c} \lvert u - \wtilde{u}(x) \rvert \d y
    = 0.
\end{equation}
In consequence, we have 
$$ \lim_{r \searrow 0} \slashint_{B_r(x)} \lvert u - \wtilde{u}(x) \rvert \d y= 0. $$
But this means that $\mathcal{H}^{n-1}$-a.e. $x \in \partial B_t$ is a Lebesgue point of $u$ for $\mathcal{L}^1$-a.e. radius $t \in (R_1, R_2)$.

Let $\alpha \in (0,1)$ be arbitrary and set $p \coloneq n/(n-1+\alpha)$. Let further $u \in \BVA_{\loc}(\R^n)$ and consider for $\varepsilon > 0$ a family of standard mollifiers $u_{\varepsilon}(x) \coloneq (\rho_{\varepsilon} \ast u)(x)$. 

Note that for each Lebesgue point $x \in \R^n$ of $u$, one has $u_{\varepsilon}(x) \to u^{\ast}(x)$ as $\varepsilon \searrow 0$, where $u^{\ast}$ is the precise representative of $u$. 

Now invoking Lemma \ref{lem:bva-slobodeckij-embedding} for $u_{\varepsilon}$ provides an element $b_{\varepsilon} \in \mathcal{N}(\A)$ such that
\begin{equation}
\label{eq:fubini-type-auxiliary-slobodeckij-estimate}
    \PARENS{\slashint_{B_r} \int_{B_r} \frac{\lvert u_{b,\varepsilon}(x) - u_{b,\varepsilon}(y) \rvert^p}{\lvert x - y \rvert^{n + \alpha p}} \d x \d y}^{\frac{1}{p}} \leq C r^{1 - \alpha} \slashint_{B_{2r}} \lvert \A u_{\varepsilon} \rvert,
\end{equation}
where $u_{b,\varepsilon} \coloneq u_{\varepsilon} - b_{\varepsilon}$. 
We note also that $b_{\varepsilon} \in C^{\infty}(\R^n;V)$ are in fact the $L^2$-orthogonal projections of $u_{\varepsilon}$ onto $\mathcal{N}(\A)$ and satisfy the $L^1$-stability estimate \autocite[Section 3.1]{BDG2017}:
\[
	\lVert b_{\varepsilon} \rVert_{L^1(B_r)} \leq C \lVert u_{\varepsilon} \rVert_{L^1(B_r)} \to \lVert u \rVert_{L^1(B_r)}.
\]
Since $\A$ is $\C$-elliptic, the nullspace $\mathcal{N}(\A)$ is of finite dimension, so one can find a subsequence $(b_{\varepsilon_j}) \subset (b_\varepsilon)$ and some $b \in \mathcal{N}(\A)$ such that $b_{\varepsilon_j} \to b$ in $\mathcal{N}(\A)$.
Consequently, denoting $u_b^{\ast}$ to be the precise representative of $u_b$, one can estimate
\begingroup
\allowdisplaybreaks
\begin{align*}
    &\int_s^r \iint_{\partial B_t \times \partial B_t} \frac{\lvert u^{\ast}_b(x) - u^{\ast}_b(y) \rvert^p}{\lvert x - y \rvert^{n-1+\alpha p}} \d \sigma_x \d \sigma_y \d t \\
    &\hspace{1.5cm}\leq \liminf_{\varepsilon_j \searrow 0} \int_s^r \iint_{\partial B_t \times \partial B_t} \frac{\lvert u_{b,\varepsilon_j}(x) - u_{b,\varepsilon_j}(y) \rvert^p}{\lvert x - y \rvert^{n-1+\alpha p}} \d \sigma_x \d \sigma_y \d t \\
    &\hspace{1.5cm}\leq C \liminf_{\varepsilon_j \searrow 0} \int_{B_r} \int_{B_r} \frac{\lvert u_{b,\varepsilon_j}(x) - u_{b,\varepsilon_j}(y) \rvert^p}{\lvert x - y \rvert^{n+\alpha p}} \d \sigma_x \d \sigma_y \tag{by \eqref{eq:general-fubini-type-theorem}} \\
    &\hspace{1.5cm}\leq C \liminf_{\varepsilon_j \searrow 0} r^n \PARENS{r^{1-\alpha} \slashint_{B_{2r}} \lvert \A u_{\varepsilon_j} \rvert}^p \tag{by \eqref{eq:fubini-type-auxiliary-slobodeckij-estimate}} \\
    &\hspace{1.5cm}\leq C r^n \PARENS{r^{1 - \alpha} \slashint_{B_{2r}} \lvert \A u \rvert}^p
\end{align*}
\endgroup
Next, towards employing the auxiliary Lemma \ref{lem:auxiliary-lebesgue-point-estimate}, consider the set
\[
    J \coloneq \{ t \in (s,r) \mid \lvert \A u \rvert(\partial B_t) = 0 \}
\]
and let $g \colon (s,r) \to \R_{\geq 0}$ be defined by
\[
    g(t) \coloneq 
    \ccases{
    \begin{aligned}
    	&\iint_{\partial B_t \times \partial B_t} \frac{\lvert u^{\ast}_b(x) - u^{\ast}_b(y) \rvert^p}{\lvert x - y \rvert^{n-1+\alpha}} \d \sigma_x \d \sigma_y, &\quad& t \in J, \\[1ex]
    	&0, &\quad& \text{else.}
    \end{aligned}
    }
\]
Then, by Step 2 and Lemma \ref{lem:auxiliary-lebesgue-point-estimate}, there is a $t \in J$ such that
\[
    g^{\ast}(t) 
    \leq \frac{2}{r - s} \int_s^r g(t) \d t 
    \leq C \frac{r^n}{r-s} \PARENS{r^{1-\alpha} \slashint_{B_{2r}} \lvert \A u \rvert}^p.
\]
Now plugging the definition of $g(t)$ in the above inequality finally produces
\[
    \PARENS{\slashint_{\partial B_t} \int_{\partial B_t} \frac{\lvert u^{\ast}_b(x) - u^{\ast}_b(y) \rvert^p}{\lvert x - y \rvert^{n-1+\alpha p}} \d \sigma_x \d \sigma_y}^{\frac{1}{p}}
    \leq C \frac{r^{\frac{n}{p}} r^{1-\alpha}}{t^{\frac{n-1}{p}} (r-s)^{\frac{1}{p}}} \slashint_{B_{2r}} \lvert \A u \rvert,
\]
which is the desired estimate \eqref{eq:fubini-type-estimate}. Note that throughout the computations, the generic constant $C$ did not depend on $u$ or $r$, completing the proof.
\end{proof}

%% file: draft-final/b-harmonic-approximation.tex
\section{$\mathcal{B}$-harmonic Approximation}

In this section, we are going to construct a $\mathcal{B}$-harmonic Approximation $h$ for a given local generalised minimser $u$. We will solve an elliptic system, which later on will be a linearisation of the Euler-Lagrange equation at a given average $(\A u)_{B}$, where $B\subset \Omega$ is a ball. We will assume that $u$ behaves nicely on the boundary $\partial B$ and we will give a precise estimation of $u-h$ in terms of the excess associated to $u$. The Ekeland variational principle allows us to prove a corresponding estimate for all minimisers close to $u$ and in the limit the estimate is inherited to $u$ itself. The precise control of $u-h$ will play a crucial role in the excess decay estimates.
\begin{theorem}
\label{theorem b-harmonic approximation}
Assuming \ref{itm:H0}, \ref{itm:H1} and \ref{itm:H2}, let $u\in \BVA_{\loc}(\Omega)$ be a generalised local minimiser of $\mathcal{F}$, let $M>0$, $q\in(1,\frac{n}{n-1})$, and let $B=B_{r}(x_{0})\Subset \Omega$ such that $u\restr{\partial B}=\Tr^{-}_{B}(u)=\Tr^{+}_{B}(u)\in \SOB^{\frac{1}{n+1},\frac{n+1}{n}}(\partial B)$. Let $a\colon \R ^{n}\to V$ be an arbitrary affine map with $|\A a|\leq M$. For $\mathcal{B}=D^{2}f(\A a) $, let $h\in \SOB^{1,\frac{n+1}{n}}(B)$ be the unique weak solution of the elliptic system
\begin{equation}
    \ccases{
    \begin{aligned}
    \A ^{*}(\mathcal{B} \A h) &= 0 &\quad& \text{in}\ B\\
     h &= u_{|\partial B}-a &\quad& \text{on}\ \partial B
    \end{aligned}
    }.
\end{equation}
Then there exists a constant $C = C(M,d_{V},d_{W},n,q,\frac{L}{\nu}) > 0$ such that
\[
    \slashint_B V \PARENS{\frac{u-a - h}{r}} \d x \leq C \PARENS{\slashint_B V(\A (u-a))}^q.
\]
\end{theorem}

\begin{proof}
We confine ourselves to merely sketching the proof, since it follows the lines of the proof \autocite[Proposition 5.4]{Gme2020}. To start with we fix some notation $\wtilde{u}=u-a$, $\wtilde{f}=(f)_{\A a}$ and  $X=\{v \in W^{\A,1}(B;\R^N) \mid \Tr(v) = \Tr(\wtilde{u})\}$.

Let $\varepsilon > 0$. There is a $u_{\epsilon} \in X$ such that

\[
   \fint _{B} \left|\frac{u_{\varepsilon}-\wtilde{u}}{r}  \right |+ \left | \fint V_{1}\left(\A u_{\varepsilon}\right)-V_{1}\left(\A \wtilde{u}\right)\right | \leq \varepsilon ^{2}, \qquad \fint_{B} \wtilde{f}(\A u_{\varepsilon}) \leq \fint_{B} \wtilde{f}(\A \wtilde{u}) +\varepsilon ^{2}.
\]
Noting that $\wtilde{u}$ is a local generalised minimiser of the functional $\mathcal{F}_{a}[\zeta;\omega]=\int_{\omega}\wtilde{f}(\A \zeta) \d x$ and taking proposition \ref{prop: continuity of functional} into account allows to apply the Ekeland variational principle \ref{lem:ekeland} to $(X,\d )$, $x=u_{\varepsilon}$ and $\mathcal{G}=\mathcal{F}_{a}$, where $\d (\zeta_{1},\zeta _{2})=\norm{\A(\zeta_{1}-\zeta_{2})}_{\LEB ^{1}(B)}$. This in conjunction with Poincaré's Inequality, yields $\wtilde{u}_{\varepsilon}\in X$ such that for all $\zeta\in \SOB ^{\A ,1}_{0}(B)$ we have
$$\Big|\int_{B} D\wtilde{f}(\A\wtilde{u}_{\varepsilon})[\A \zeta ]\d x \Big | \leq \varepsilon\int_{B}|\A \zeta |\d x\ \textnormal{and}\ \norm{u_{\varepsilon}-\wtilde{u}_{\varepsilon}}_{\SOB^{\A,1}(B)}\leq C\varepsilon (r^{n}+r^{n+1})$$
for a constant $C>0$ independent of $u_{\varepsilon}, \wtilde{u}_{\varepsilon}$ and $r$. Writing $\mathcal{B}[\wtilde{u}_{\varepsilon},\cdot]=\Lambda [\A \wtilde{u}_{\varepsilon}, \cdot] +D\wtilde{f}(\A \wtilde{u}_{\varepsilon})$ and applying the pointwise estimate \ref{auxiliary estimate linearisation} to the first term then yields for some constant $C(M,L)>0$ and all $\zeta \in \SOB^{\A, 1}_{0}(B)$
\begin{equation}
\label{auxiliary estimate harmonic approximation}
    \Big |\int_{B} \mathcal{B}[\A \wtilde{u}_{\varepsilon},\A \zeta]\d x \Big |\leq C \int_{B}V_{1}(\A \wtilde{u}_{\varepsilon})|\A \zeta |\d x+\varepsilon\int_{B}|\A \zeta |\d x .
\end{equation}

At this stage, we scale things to the unit ball $\mathbb{B}$: For a measurable map $p$ defined on $B$ we put $S[p](x)=r^{-1}p(x_{0}+rx)$. Furthermore, we put $\Psi_{\varepsilon}=S[\wtilde{u}_{\varepsilon}-h]$ and  $U_{\varepsilon}=S[\wtilde{u}_{\varepsilon}]$. We will now truncate the map $\Psi_{\varepsilon}$. To this aim, we put
\begin{equation*}
    \textbf{T}(w)=
    \begin{cases}
     w,\ &\norm{w}\leq 1\\
    \frac{w}{\norm{w}},\ &\norm{w}>1
    \end{cases}.
\end{equation*}
We fix $p>n$ and let $\Phi_{\varepsilon}\in \SOB^{1,\infty}\cap \SOB^{1,p}_{0})(\mathbb{B};V)$ be the solution of the elliptic system
\begin{equation}
    \ccases{
    \begin{aligned}
    \A ^{*}(\mathcal{B} \A \Phi_{\varepsilon}) &= \textbf{T}\circ \Psi _{\varepsilon} &\quad& \text{in}\ \mathbb{B}\\
     \Phi_{\varepsilon} &= 0&\quad& \text{on}\ \mathbb{S}
    \end{aligned}
    }.
\end{equation}
Now testing the system with $\Psi _{\varepsilon}\in \SOB^{\A ,1}_{0}(\mathbb{B})$ and exploiting \ref{auxiliary estimate harmonic approximation} yields the key estimation
\begin{equation*}
    \begin{split}
        \int_{\mathbb{B}}V_{1}(\Psi_{\varepsilon})\d x&\leq \int_{\mathbb{B}}\langle \textbf{T}\circ \Psi _{\varepsilon}, \Psi _{\varepsilon} \rangle \d x \\
        &=\int_{\mathbb{B}}\mathcal{B}[\A \Phi _{\varepsilon},\A \Psi _{\varepsilon}]\\
        &\leq C \big ( \int_{\mathbb{B}}V_{1}(\Psi _{\varepsilon}) \d x+ \varepsilon\big)\norm{\textbf{T}\circ \Psi _{\varepsilon}}_{\LEB ^{p}(B)}\\
        & \leq C \big ( \int_{\mathbb{B}}V_{1}(U _{\varepsilon}) \d x+ \varepsilon\big)\big(\int_{\mathbb{B}}V_{1}(\Psi _{\varepsilon})\d x \big)^{\frac{1}{p}}
    \end{split}
\end{equation*}
for a constant $C = C(M,d_{V},d_{W},n,q,\frac{L}{\nu})$. Note that in the penultimate step we have exploited the estimate \ref{estimate elliptic system homogeneous with boundary datum}. Dividing by $(\int_{\mathbb{B}}V_{1}(\Psi _{\varepsilon})\d x )^{\frac{1}{p}}$, sending $\varepsilon \to 0$, scaling back to the ball $B$ and setting $q=p^{\prime}$ concludes the proof.
\end{proof}

%% file: draft-final/excess-decay.tex
\section{Excess Decay}
In this section, we will display the most important step towards proving an excess decay for a given local generalised minimiser $u$. First, we will invoke Caccioppoli's Inequality and then, using our Fubini-type theorem for $\BVA$-maps, we will construct a $\mathcal{B}$-harmonic approximation. Then we will show separately that $h$ and $u-h$ have a good decay. Here, the excess of $u$ is defined as follows:
$$\excess(u,x,r)\coloneqq \int_{B_{r}(x)}V_{1}\left(\A u- (\A u)_{B_{r}(x)}\right)\ \quad\textnormal{and}\quad \excessn(u,x,r)\coloneqq\frac{\excess(u,x,r)}{\mathcal{L}^{n}(B_{r}(x)}. $$

\begin{lemma}
\label{lemma: prelima deacy}
Assuming \ref{itm:H0}, \ref{itm:H1} and \ref{itm:H2}, let $u\in\BVA_{\loc}(\Omega)$ be a generalised local minimiser of $\mathcal{F}$. Furthermore, let $M>0$ and $q\in (1,\frac{n}{n-1})$. There exists a constant $C(M,\A, q,d_{V},d_{W}, \frac{L}{\nu})>0$ with the property: If we have $B=B(x_{0},R)\Subset \Omega$, $|(\A u)_{B}|<M$ and $\fint _{B}|\A u -(\A u)_{B}|\leq 1$, then for all $\sigma \in (0,1]$ we have
\begin{equation}
    \excessn(u;x_{0},\sigma R)\leq c\Big(\sigma ^{2}+\sigma ^{-n-2}\big(\excessn(u;x_{0},R)\big)^{q-1} \Big)\excessn(u;x_{0},R).
\end{equation}
\end{lemma}

\begin{proof}
We are going to set up the proof as follows:
\begin{itemize}
    \item Due to Lemma \ref{lem: mean value} we find $a\in\mathcal{L}(\R ^{n};V)$ such that $\A a =(\A u)_{B}$. We put $\wtilde{u}=u-a$, $\wtilde{f}=(f)_{(\A u)_{B}}$ and $\mathcal{B}=D^{2}\wtilde{f}(0)$.
    \item Due to Theorem \ref{thm:fubini-type} there exists a radius $r\in (\frac{38R}{40},\frac{39R}{40})$ such that $\wtilde{u}\restr{\partial B_r(x_0)}$ is well-defined and belongs to the space $\SOB^{\frac{1}{n+1}, \frac{n+1}{n}}(\partial \tilde{B})$ with $\tilde{B}=B_{r}$. Furthermore, for some $b\in\mathcal{N}(\A)$, we have
    \begin{equation}
    \label{auxiliary estimate epsiolon decay harmonic map}
        [\wtilde{u}_{b}]_{\SOB^{\frac{1}{n+1}, \frac{n+1}{n}}(\partial \tilde{B})}\leq C\fint_{B_{2r}(x_{0})}|\A \wtilde{u}|\leq C \excessn(u;x_{0},R)
    \end{equation}
    for a constant $C(\A, n, \alpha)$.
    \item Let $h$ be the solution of the elliptic system 
    \begin{equation*}
    \ccases{
    \begin{aligned}
    \A ^{*}(\mathcal{B} \A h) &= 0 &\quad& \text{in}\ \tilde{B}\\
     h &= u_{|\partial \tilde{B}}-a &\quad& \text{on}\ \partial \tilde{B}
    \end{aligned}
    }.
    \end{equation*}
    \item In view of Corollary \ref{corollary: elliptic estimates}, we put $a_{0}=a+A_{h-b}$ and let $\sigma \in (0,\frac{1}{10})$ be arbitrary but fixed. Since we have $\fint_{B}V_{1}(\A \wtilde{u})\leq 1$, we may estimate by means of Lemma \ref{lem:linearisation and refernece integrand} and Jensen's Inequality
    \begin{equation}
    \label{auxiliary estimate epsilon decay}
        V_{1}\Big(\sigma \fint_{B}\A \wtilde{u}\Big)\leq C \sigma ^{2}\Big(\fint_{B}\A \wtilde{u}\Big)^{2}\leq C \sigma ^{2}V_{1}\Big(\fint_{B}\A \wtilde{u}\Big) \leq C\sigma ^{2}\excessn(u;x_{0},R).
    \end{equation}
    Combining the estimates \ref{auxiliary estimate epsilon decay}, \ref{auxiliary estimate harmonic map} and \ref{auxiliary estimate epsiolon decay harmonic map} yields 
    \begin{equation}
    \label{aux estimate harmonic map final}
        \int _{B_{2\sigma R}}V_{1}\Big(\frac{h-b-A_{h-b}}{2\sigma R}\Big)\leq CR^{n}\sigma ^{n+2}\excessn(u;x_{0},R).
    \end{equation}
\end{itemize}
We note that we have due to Corollary \ref{corollary: elliptic estimates} $|\A a_{0}|\leq M_{0}+C\fint_{B_{2r}}|\A \wtilde{u}|\leq M_{0}+C\coloneqq m $ for a constant $C(m,\A, d_{V},d_{W},\frac{L}{\nu})$. Especially, $u-b$ is a generalised local minimiser as well. By Caccioppoli's Inequality, Proposition \ref{prop:caccioppoli}, we estimate
\begin{equation*}
 \begin{split}
    \excess (u;x_{0},\sigma R)&\leq C\int _{B_{2\sigma R}}V_{1}\Big(\frac{u-b-a_{0}}{\sigma R}\Big)\\
    &\leq C \Big\{\sigma ^{-2}\int _{B_{\frac{r}{2}}}V_{1}\Big(\frac{\wtilde{u}-h}{r}\Big)+\int _{B_{2\sigma R}}V_{1}\Big(\frac{h-b-A_{h-b}}{2\sigma R}\Big)\Big\}\\
    &\leq C\Big\{\sigma ^{-2} R^{n}(\excessn(u,x_{0}, R))^{q}+R^{n}\sigma ^{n+2}\excess(u,x_{0},2 R)\Big\}
\end{split}
\end{equation*}
for a constant $C(m,\A, d_{V},d_{W},\frac{L}{\nu})$. In the second step, we have exploited Theorem~\ref{theorem b-harmonic approximation} and the estimate \ref{aux estimate harmonic map final}.
\end{proof}

\begin{proposition}
\label{proposition: epsilon decay}
Assuming \ref{itm:H0}, \ref{itm:H1} and \ref{itm:H2}, let $u\in\BVA_{\loc}(\Omega)$ be a generalised local minimiser of $\mathcal{F}$, where $f$ satisfies \ref{itm:H1} and \ref{itm:H2}. Let $\alpha \in (0,1)$ and $M>0$. Then there exist constants $\gamma(M,\A, \alpha,d_{V},d_{W} , \frac{L}{\nu})$ and $\varepsilon (M,\A, \alpha,d_{V},d_{W} , \frac{L}{\nu}) >0 $ with the property: If we have $B=B(x_{0},R)\Subset \Omega$, $|(\A u)_{B}|<M$ and $\excessn (u;x_{0};R)< \varepsilon$, then we have for all $\vartheta\in (0,1)$:
\begin{equation}
  \excessn (u;x_{0};\vartheta R)\leq \gamma  \vartheta ^{\alpha}\excessn (u;x_{0};R).
\end{equation}
\end{proposition}
The proof is analogous to the proof of \autocite[Proposition 4.8]{GK2019} or \autocite[Proposition 5.7]{Gme2020} since we already have established Lemma \ref{lemma: prelima deacy}.
At this stage, we sketch the proof of the main theorem. We will pay special attention to the final step, which needs to be modified in regards to \autocite{Gme2020}:
\begin{proof}
Using \autocite[1.6.1, Theorem 1; 1.6.2, Theorem 3]{EG1992}, we observe that $\mathcal{L}^{n}(\Sigma _{u})=0$.

Let $x_{0}\in \Omega\setminus \Sigma_{u}$ and $\alpha \in (0,1)$. For $M>0$, we denote by $\gamma_{M}$ and $\varepsilon_{M}$ the constants determined by Proposition \ref{proposition: epsilon decay}. Then there exists $M>0$ and a radius $R>0$ with $B_{R}(x_{0})\Subset \Omega$ such that for all $x\in \tilde{B}=B_{\frac{R}{2}}(x_{0})$ we have
$$\left|(\A u _{\tilde{B}})\right|\leq M \quad\textnormal{and}\quad \excessn\Big(u;x;\frac{R}{2}\Big)\leq  \varepsilon_{M}.$$
This can be inferred from a similar argument to the one in \autocite[32]{Gme2020}. In particular, we then have for all $x\in B=B_{\frac{R}{4}}(x_{0})$ and all $0<r<\frac{R}{2}$:
$$\excessn(u;x;r)\leq \gamma _{M}\varepsilon_{M}\Big(\frac{2}{R}\Big)^{\alpha}r^{\alpha}.$$
From this decay, it can be deduced by a simple covering argument that the measure $|\A u \mres{B}|$ is absolutely continuous with respect to the Lebesgue measure.

Towards the Hölder continuity of the distributional gradient $Du$, we invoke Campanato's characterisation \autocite[Theorem 5.1]{Cam1964}: Let $k=\dim_{\R} \mathcal{N}(\A )\geq 1$ and $q$ be a linear $V$-valued polynomial such that $\A q =(\A u)_{B(x,r)}$. Furthermore, we put $p=q+\Pi(u-q)\in \mathcal{P}_{k}(\R ^{N})$. Now estimate for all $x\in B$ and all $0<r<\frac{R}{2}$
$$\fint_{B(y,r)}|u-p|\d x\leq C_{Poin} r \fint_{B(y,r)}|\A (u-q)|\d x\leq C r^{1+\alpha}.$$
\end{proof}

%% file: res/bibliography.bib
@Article{GK2019,
	Author    = {Gmeineder, Franz and Kristensen, Jan},
	Year      = {2019},
	Title     = {Partial Regularity for BV Minimizers},
	Number    = {\ no. 3},
	Pages     = {1429--1473},
	Volume    = {232},
	FJournal  = {{Archive for Rational Mechanics and Analysis}},
	Journal   = {Arch. Ration. Mech. Anal.},
	Publisher = {Springer},
	Zbl       = {1411.26016}
}

@Article{CK2017,
	Author    = {Chen, Chuei Yee and Kristensen, Jan},
	Year      = {2017},
	Title     = {On coercive variational integrals.},
	Volume    = {153},
	FJournal  = {Nonlinear Anal.},
	Journal   = {Nonlinear Anal.},
	Pages     = {213--229}
}

@Article{BDG2017,
	Author      = {Breit, Dominic and Diening, Lars and Gmeineder, Franz},
	Year        = {2017},
	Title       = {On the Trace Operator for Functions of Bounded $\mathbb{A}$-Variation},
	FJournal  = {Anal. PDE},
	Journal   = {Anal. PDE},
	Volume    = {13},
	Number    = {\ no. 2},
	Pages     = {559–-594},
	Year      = {2020},
}

@Article{GR2019,
	Author    = {Gmeineder, Franz and Rai\c{t}\u{a}, Bogdan},
	Title     = {Embeddings for \(\mathbb{A}\)-weakly differentiable functions on domains.},
	FJournal  = {Journal of Functional Analysis},
	Journal   = {J. Funct. Anal.},
	Volume    = {277},
	Number    = {\ no. 12},
	%Pages     = {33},
	Year      = {2019},
	Publisher = {Elsevier},
	Address   = {Amsterdam},
	%Language  = {English},
	MSC2010   = {46E35 26D10 35J48}
}

@Article{Kal1994,
	Author    = {Ka{\l}amajska, Agnieszka},
	Title     = {Pointwise multiplicative inequalities and Nirenberg type estimates in weighted Sobolev spaces.},
	FJournal  = {Studia Mathematica},
	Journal   = {Stud. Math.},
	Volume    = {108},
	Number    = {\ no. 3},
	Pages     = {275--290},
	Year      = {1994},
	Publisher = {Polish Academy of Sciences (Polska Akademia Nauk - PAN), Institute of Mathematics (Instytut Matematyczny)},
	Address   = {Warsaw},
	%Language  = {English},
	MSC2010   = {46E35 26D10 31B10},
	Zbl       = {0819.46021}
}

@Book{Giu2003,
	Author    = {Giusti, Enrico},
	Title     = {Direct methods in the calculus of variations.},
	%Pages     = {vii + 403},
	Year      = {2003},
	Publisher = {World Scientific},
	Address   = {Singapore},
	%Language  = {English},
	MSC2010   = {49-01 49J10 49J45 49N60 35B65 35Jxx},
	Zbl       = {1028.49001}
}

@Article{Gme2020,
	Author   = {Gmeineder, Franz},
	Year     = {2021},
	Title    = {Partial Regularity for Symmetric Quasiconvex Functionals on BD},
	FJournal = {Journal de Math\'ematiques Pures et Appliqu\'ees},
	Journal	 = {J. Math. Pures Appl.},
	Volume   = {145},
	Number   = {\ no. 9},
	Pages    = {83--129}
}

@Article{PR2016,
	Author    = {De Philippis, Guido and Rindler, Filip},
	Title     = {On the structure of \({\mathcal A}\)-free measures and applications.},
	FJournal  = {Annals of Mathematics. Second Series},
	Journal   = {Ann. Math. (2)},
	Volume    = {184},
	Number    = {\ no. 3},
	Pages     = {1017--1039},
	Year      = {2016},
	Publisher = {Princeton University, Mathematics Department},
	Address   = {Princeton, NJ},
	%Language  = {English},
	MSC2010   = {49Q15 35D30 28B05 49Q20},
	Zbl       = {1352.49050}
}

@Article{FM1999,
	Author    = {Fonseca, Irene and M\"uller, Stefan},
	Title     = {\(\mathcal A\)-quasiconvexity, lower semicontinuity, and Young measures.},
	FJournal  = {SIAM Journal on Mathematical Analysis},
	Journal   = {SIAM J. Math. Anal.},
	Volume    = {30},
	Number    = {\ no. 6},
	Pages     = {1355--1390},
	Year      = {1999},
	Publisher = {Society for Industrial and Applied Mathematics (SIAM)},
	Address   = {Philadelphia, PA},
	%Language  = {English},
	MSC2010   = {49J45},
	Zbl       = {0940.49014}
}

@Article{VS2013,
	Author    = {Van Schaftingen, Jean},
	Title     = {Limiting Sobolev inequalities for vector fields and canceling linear differential operators.},
	FJournal  = {Journal of the European Mathematical Society (JEMS)},
	Journal   = {J. Eur. Math. Soc. (JEMS)},
	Volume    = {15},
	Number    = {\ no. 3},
	Pages     = {877--921},
	Year      = {2013},
	Publisher = {European Mathematical Society (EMS) Publishing House},
	Address   = {Zurich},
	%Language  = {English},
	MSC2010   = {46E35 26D10 42B20},
	Zbl       = {1284.46032}
}

@Article{Orn1962,
	Author    = {Ornstein, Donald},
	Title     = {A non-inequality for differential operators in the \(L_ 1\) norm.},
	FJournal  = {Archive for Rational Mechanics and Analysis},
	Journal   = {Arch. Ration. Mech. Anal.},
	Volume    = {11},
	Pages     = {40--49},
	Year      = {1962},
	Publisher = {Springer},
	Address   = {Berlin/Heidelberg},
	%Language  = {English},
	Zbl       = {0106.29602}
}

@Book{Bec2016,
	Author    = {Beck, Lisa},
	Title     = {Elliptic regularity theory. A first course.},
	FSeries   = {Lecture Notes of the Unione Matematica Italiana},
	Series    = {Lect. Notes Unione Mat. Ital.},
	Volume    = {19},
	%Pages     = {xii + 201},
	Year      = {2016},
	Publisher = {Sprin\-ger},
	Address   = {Cham/Heidelberg/New York/Dordrecht/London},
	%Language  = {English},
	MSC2010   = {35-01 35J62 49K20 35B65},
	Zbl       = {1346.35001}
}

@Article{Min2006,
	Author    = {Mingione, Giuseppe},
	Title     = {Regularity of minima: an invitation to the dark side of the calculus of variations.},
	FJournal  = {Applications of Mathematics},
	Journal   = {Appl. Math., Praha},
	Volume    = {51},
	Number    = {\ no. 4},
	Pages     = {355--425},
	Year      = {2006},
	Publisher = {Czech Academy of Sciences},
	Address   = {Institute of Mathematics, Prague},
	%Language  = {English},
	MSC2010   = {49N60 35J20 35J70},
	Zbl       = {1164.49324}
}

@Article{Hil1900,
	Author  = {Hilbert, David},
	Title   = {Mathematische Probleme. Vortrag, gehalten auf dem internationalen Mathematiker-Kongreß zu Paris 1900},
	Journal = {Nachrichten von der Kö\-nig\-li\-chen Gesellschaft der Wissenschaften zu Göttingen. Mathematisch-Physikalische Klasse},
	Volume  = {3},
	Edition = {1},
	Pages   = {253--297},
	Year    = {1900}
}

@Article{Gio1957,
	Author       = {De Giorgi, Ennio},
	Title        = {Sulla differenziabilit\`a e l'analiticit\`a delle estremali degli integrali multipli regolari.},
	Year         = {1957},
	FJournal  = {Mem. Accad. Sci. Torino. Cl. Sci. Fis. Mat. Nat. (3)},
	Journal   = {Mem. Accad. Sci. Torino. Cl. Sci. Fis. Mat. Nat. (3)},
	Volume    = {3},
	Pages     = {25--43},
	Language     = {Italian}
}

@Article{Nash1958,
	Author    = {Nash, John F.},
	Title     = {Continuity of solutions of parabolic and elliptic equations.},
	FJournal  = {American Journal of Mathematics},
	Journal   = {Amer. J. Math.},
	Volume    = {80},
	Pages     = {931--954},
	Year      = {1958},
	Publisher = {Johns Hopkins University Press},
	Address   = {Baltimore, MD},
	%Language  = {English},
	Zbl       = {0096.06902}
}

@Article{Mo1960,
	Author    = {Moser, Jürgen},
	Title     = {A new proof of De Giorgi's theorem concerning the regularity problem for elliptic differential equations.},
	FJournal  = {Comm. Pure Appl. Math.},
	Journal   = {Comm. Pure Appl. Math.},
	Volume    = {13},
	Pages     = {457-–468},
	Year      = {1960},
}

@Article{Gio1968,
	Author    = {De Giorgi, Ennio},
	Title     = {Un esempio di estremali discontinue per un problema variazionale di tipo ellittico.},
	FJournal  = {Bollettino della Unione Matematica Italiana. Series IV},
	Journal   = {Boll. Unione Mat. Ital., IV. Ser.},
	Volume    = {1},
	Pages     = {135--137},
	Year      = {1968},
	Publisher = {Zanichelli},
	Address   = {Bologna},
	Language  = {Italian},
	Zbl       = {0155.17603}
}

@Article{Nec1977,
	Author    = {Ne\v{c}as, Jind\v{r}ich},
	Title     = {Example of an irregular solution to a nonlinear elliptic system with analytic coefficients and conditions for regularity.},
	Year      = {1977},
	%Language = {English},
	Journal   = {Theor. Nonlin. Oper., Constr. Aspects, Proc. 4th Int. Summer School, 1975},
	Pages = {197--206},
	Address   = {Berlin},
	Publisher = {Akademie-Verlag},
	MSC2010   = {35J60 35D10},
	Zbl       = {0372.35031}
}

@Book{Gio1961,
	Author    = {De Giorgi, Ennio},
	Title     = {Frontiere orientate di misura minima.},
	FSeries   = {Seminario di Matematica della Scuola Normale Superiore di Pisa},
	Series    = {Sem. Mat. Scuola Norm. Sup.},
	Publisher = {Editrice Tecnico Scientifica},
	Address   = {Pisa},
	Year      = {1961},
	Language  = {Italian}
}

@Article{Fr1970,
	Author    = {Frehse, Jens},
	Title     = {Una generalizzazione di un controesempio di De Giorgi nella teoria delle equazioni ellitiche.},
	FJournal  = {Boll. Un. Mat. Ital.},
	Journal   = {Boll. Un. Mat. Ital.},
	Volume    = {4},
	Number    = {\ no. 3},
	Pages     = {998--1002},
	Year      = {1970},
}

@Article{CP1996,
	Author    = {Carozza, Menita and Passarelli di Napoli, Antonia},
	Title     = {A regularity theorem
for minimisers of quasiconvex integrals: the case 1 < p < 2.},
	FJournal  = {Proc. Roy.
Soc. Edinburgh Sect. A},
	Journal   = {Proc. Roy.
Soc. Edinburgh Sect. A},
	Volume    = {126},
	Number    = {\ no. 6},
	Pages     = {1181–-1199},
	Year      = {1996},
}

@Article{DG2000,
	Author    = {Duzaar, F. and Grotowski, J.F.},
	Title     = {Optimal interior partial regularity for nonlinear elliptic systems: The method of A-harmonic approximation.},
	FJournal  = {Manuscr. Math.},
	Journal   = {Manuscr. Math.},
	Volume    = {103},
	Pages     = {267--298},
	Year      = {2000},
}

@Article{DLSV2012,
	Author    = {Diening, Lars and Lengeler, Daniel and Stroffolini, Bianca and Verde, Anna},
	Title     = {Partial regularity for minimizers of quasi-convex functionals with general growth.},
	FJournal  = {SIAM J. Math. Anal.},
	Journal   = {SIAM J. Math. Anal.},
	Volume    = {44},
	Number    = {\ no. 5},
	Pages     = {3594--3616},
	Year      = {2012}
}

@Article{Alm1968,
	Author    = {Almgren, Frederick J. Jr.},
	Title     = {Existence and regularity almost everywhere of solutions to elliptic variational problems among surfaces of varying topological type and singularity structure.},
	FJournal  = {Annals of Mathematics. Second Series},
	Journal   = {Ann. of Math. },
	Volume    = {87},
	Number    = {\ no. 2},
	Pages     = {321--391},
	Year      = {1968},
	Publisher = {Princeton University, Mathematics Department},
	Address   = {Princeton, NJ},
	%Language  = {English},
	Zbl       = {0162.24703}
}

@Article{Eva1986,
	Author    = {Evans, Lawrence C.},
	Title     = {Quasiconvexity and partial regularity in the calculus of variations.},
	FJournal  = {Archive for Rational Mechanics and Analysis},
	Journal   = {Arch. Ration. Mech. Anal.},
	Volume    = {95},
	Number    = {\ no. 3},
	Pages     = {227--252},
	Year      = {1986},
	Publisher = {Springer},
	Address   = {Berlin/Heidelberg},
	%Language  = {English},
	MSC2010   = {49J45 35D10 35J60 26B25},
	Zbl       = {0627.49006}
}

@Article{Mo1952,
	Author    = {Morrey, Charles B., Jr.},
	Title     = {Quasi-convexity and the lower semicontinuity of multiple integrals.},
	FJournal  = {Pacific J. Math.},
	Journal   = {Pacific J. Math.},
	Volume    = {2},
	Pages     = {25-–53},
	Year      = {1952},
}

@Article{Al1993,
	Author    = {Alberti, Giovanni},
	Title     = {Rank one property for derivatives of functions with boun\-ded variation.},
	FJournal  = {Proc. Roy. Soc. Edinburgh Sect. A 123},
	Journal   = {Proc. Roy. Soc. Edinburgh Sect. A 123.},
	Number    = {\ no. 2},
	Pages     = {239-–274},
	Year      = {1993},
}

@Article{AD1992,
	Author    = {Ambrosio, Luigi and Dal Maso, Gianni},
	Title     = {On the relaxation in $\BV (\Omega;\R^{m})$ of quasi-convex integrals.},
	FJournal  = {J. Funct. Anal.},
	Journal   = {J. Funct. Anal.},
	Volume    = {109},
	Number    = {\ no. 1},
	Pages     = {76--97},
	Year      = {1992},
}

@Article{AF1987,
	Author    = {Acerbi, Emilio and Fusco, Nicola},
	Title     = {A regularity theorem for minimizers of quasiconvex integrals.},
	FJournal  = {Archive for Rational Mechanics and Analysis},
	Journal   = {Arch. Ration. Mech. Anal.},
	Volume    = {99},
	Number    = {\ no. 3},
	Pages     = {261--281},
	Year      = {1987},
	Publisher = {Springer},
	Address   = {Berlin/Heidelberg},
	%Language  = {English},
	MSC2010   = {49J45 35D10 35J60 26B25},
	Zbl       = {0627.49007}
}

@Article{CFM1998,
	Author    = {Carozza, Menita and Fusco, Nicola and Mingione, Giuseppe},
	Title     = {Partial regularity of minimizers of quasiconvex integrals with subquadratic growth.},
	FJournal  = {Annali di Matematica Pura ed Applicata. Serie Quarta},
	Journal   = {Ann. Mat. Pura Appl. (4)},
	Volume    = {175},
	Pages     = {141--164},
	Year      = {1998},
	Publisher = {Springer},
	Address   = {Berlin/Heidelberg},
	% Fondazione Annali di Matematica Pura ed Applicata c/o Dipartimento di Matematica ``U. Dini'', Firenze (Publisher & Address)
	%Language  = {English},
	MSC2010   = {49N60 49J45},
	Zbl = {0960.49025}
}

@Article{AG1988,
	Author    = {Anzellotti, Gabriele and Giaquinta, Mariano},
	Title     = {Convex functionals and partial regularity.},
	FJournal  = {Archive for Rational Mechanics and Analysis},
	Journal   = {Arch. Ration. Mech. Anal.},
	Volume    = {102},
	Number    = {\ no. 3},
	Pages     = {243--272},
	Year      = {1988},
	Publisher = {Springer},
	Address   = {Berlin/Heidelberg},
	%Language  = {English},
	MSC2010   = {49J45 35D10 46E35 26B30},
	Zbl       = {0658.49005}
}

@Article{Sch2014,
	Author    = {Schmidt, Thomas},
	Title     = {Partial regularity for degenerate variational problems and image restoration models in BV.},
	FJournal  = {Indiana University Mathematics Journal},
	Journal   = {Indiana Univ. Math. J.},
	Volume    = {63},
	Number    = {\ no. 1},
	Pages     = {213--279},
	Year      = {2014},
	Publisher = {Indiana University, Department of Mathematics},
	Address   = {Bloomington, IN},
	%Language  = {English},
	MSC2010   = {49N60 26B30 46E30 94A08},
	Zbl       = {1319.49058}
}

@Book{EG1992,
	Author    = {L.C., Evans and R.F., Gariepy},
	Title     = {Measure Theory and Fine Properties of Functions},
	FJournal  = {Studies in Advanced Mathematics},
	Journal   = {Studies in Advanced Mathematics},
	Year      = {1992},
	Publisher = {CRC Press},
}

@Article{Cam1964,
	Author    = {Campanato, Sergio},
	Title     = {Proporietà di una famiglia
di spazi funzionali.},
	FJournal  = { Ann. Scuola Norm. Sup. Pisa Cl. Sci.},
	Journal   = { Ann. Scuola Norm. Sup. Pisa Cl. Sci.},
	Volume    = {3},
	Pages     = {137--160},
	Year      = {1964},
	Publisher = {},
	Address   = {},
	Language  = {Italian},
	MSC2010   = {},
	Zbl       = {}

}

@Misc{CG2020,
	Author      = {Conti, Sergio and Gmeineder, Franz},
	Year        = {2020},
	Title       = {$\mathcal{A}$-Quasiconvexity and Partial Regularity.},
	Eprint      = {2009.13820},
	%Eprintclass = {math.AP},
	Eprinttype  = {arXiv}
}

@Article{KK2011,
	Author    = {Kirchheim, Bernd and Kristensen, Jan},
	Title     = {Automatic convexity of rank-1 convex functions.},
	FJournal  = {C. R. Math. Acad. Sci. Paris},
	Journal   = {C. R. Math. Acad. Sci. Paris},
	Volume    = {349 no. 7--8},
	Pages     = {407-–409},
	Year      = {2011}
	%Language  = {English},
}

@Article{CFM2005,
	Author    = {Conti, Sergio and Faraco, Daniel and Maggi, Francesco},
	Title     = {A new approach to counterexamples to $\LEB ^{1}$ estimates: Korn's inequality, geometric rigidity, and regularity for gradients of separately convex functions.},
	FJournal  = {Arch. Ration. Mech. Anal.},
	Journal   = {Arch. Ration. Mech. Anal.},
	Volume    = {175},
	Number    = {\ no. 2},
	Pages     = {287--300},
	Year      = {2005}
	%Language  = {English}
}

@Article{DM2009,
	Author    = {Duzaar, Frank and Mingione, Giuseppe},
	Title     = {Harmonic type approximation lemmas.},
	FJournal  = {J. Math. Anal. Appl.},
	Journal   = {J. Math. Anal. Appl.},
	Volume    = {352},
	Number    = {\ no. 1},
	Pages     = {301-–335},
	Year      = {2009}
	%Language  = {English},
}

@Article{Ek1974,
	Author    = {Ekeland, Ivar},
	Title     = {On the variational principle.},
	FJournal  = {J. Math. Anal. Appl.},
	Journal   = {J. Math. Anal. Appl.},
	Volume    = {47},
	Pages     = {324-–353},
	Year      = {1974}
	%Language  = {English},
}

@Article{Mi1956,
	Author    = {Mihlin, S. G.},
	Title     = {On the multipliers of Fourier integrals.},
	FJournal  = {Dokl. Akad. Nauk SSSR (N.S.)},
	Journal   = {Dokl. Akad. Nauk SSSR (N.S.)},
	Volume    = {109},
	Pages     = {701–-703},
	Year      = {1956},
	Language  = {Russian}
}

@Article{CZ1956,
	Author    = {Caldéron, A. P. and Zygmund, A.},
	Title     = {On singular integrals.},
	FJournal  = {Amer. J. Math.},
	Journal   = {Amer. J. Math.},
	Volume    = {78},
	Pages     = {289-–-309},
	Year      = {1956}
	%Language  = {English},
}

@Article{BM1984,
	Author    = {Ball, J. M. and Murat, F.},
	Title     = {$\SOB^{1,p}$-quasiconvexity and variational problems for multiple integrals.},
	FJournal  = {J. Funct. Anal.},
	Journal   = {J. Funct. Anal.},
	Volume    = {58},
	Number    = {\ no. 3},
	Pages     = {225–-253},
	Year      = {1984}
	%Language  = {English},
}

@Article{AF1984,
	Author    = {Acerbi, Emilio and Fusco, Nicola},
	Title     = {Semicontinuity problems in the calculus of variations.},
	FJournal  = {Arch. Rational Mech. Anal.},
	Journal   = {Arch. Rational Mech. Anal.},
	Volume    = {86},
	Number    = {\ no. 2},
	Pages     = {125–-145},
	Year      = {1984}
	%Language  = {English},
}

@Article{KR2010,
	Author    = {J. Kristensen and F. Rindler},
	Title     = {Relaxation of signed integral functionals in $\BV$.},
	FJournal  = {Calc. Var. Partial Differential Equations},
	Journal   = {Calc. Var. Partial Differential Equations},
	Volume    = {37},
	Number    = {\ no. 1--2},
	Pages     = {29–-62},
	Year      = {2010}
	%Language  = {English},
}

@Article{Min2008,
    Author    = {Mingione, Giuseppe},
    Title     = {Singularities of minima: a walk on the wild side of the calculus of variations},
    FJournal  = {{Journal of Global Optimization}},
    Journal   = {J. Glob. Optim.},
    %ISSN      = {0925-5001},
    Volume    = {40},
    Number    = {\ no. 1--3},
    Pages     = {209--223},
    Year      = {2008},
    Publisher = {Springer US, New York, NY},
    %Language  = {English},
    %DOI       = {10.1007/s10898-007-9226-1},
    %MSC2010   = {49N60 26A33},
    %Zbl       = {1295.49025}
}

@Article{ADR2020,
    Author    = {Arroyo-Rabasa, Adolfo and De Philippis, Guido and Rindler, Filip},
    Title     = {Lower semicontinuity and relaxation of linear-growth integral functionals under PDE constraints},
    FJournal  = {{Advances in Calculus of Variations}},
    Journal   = {Adv. Calc. Var.},
    %ISSN = {1864-8258},
    Volume.   = {13},
    Number    = {\ no. 3},
    Pages     = {219--255},
    Year      = {2020},
    Publisher = {De Gruyter, Berlin},
    %Language  = {English},
    %DOI = {10.1515/acv-2017-0003},
    %MSC2010 = {49J45 35J50 28B05},
    %Zbl = {1445.49006}
}

@Article{Gme2020-dirichlet,
    Author = {Gmeineder, Franz},
    Title = {The regularity of minima for the Dirichlet problem on BD},
    FJournal = {{Archive for Rational Mechanics and Analysis}},
    Journal = {Arch. Ration. Mech. Anal.},
    %ISSN = {0003-9527},
    Volume = {237},
    Number = {\ no. 3},
    Pages = {1099--1171},
    Year = {2020},
    Publisher = {Springer, Berlin/Heidelberg},
    %Language = {English},
    %DOI = {10.1007/s00205-020-01507-5},
    %MSC2010 = {49N60 35A15},
    %Zbl = {1442.49048}
}

@Article{DGK2005,
    Author = {Duzaar, Frank and Grotowski, Joseph F. and Kronz, Manfred},
    Title = {Regularity of almost minimizers of quasi-convex variational integrals with subquadratic growth},
    FJournal = {{Annali di Matematica Pura ed Applicata. Serie Quarta}},
    Journal = {Ann. Mat. Pura Appl.},
    %ISSN = {0373-3114},
    Volume = {184},
    Number = {\ no. 4},
    Pages = {421--448},
    Year = {2005},
    Publisher = {Springer, Berlin/Heidelberg; Fondazione Annali di Matematica Pura ed Applicata c/o Dipartimento di Matematica ``U. Dini'', Firenze},
    %Language = {English},
    %DOI = {10.1007/s10231-004-0117-5},
    %MSC2010 = {49N60 26B25},
    %Zbl = {1223.49040}
}

@Article{Spe1969,
    Author = {Spencer, D. C.},
    Title = {Overdetermined systems of linear partial differential equations},
    FJournal = {{Bulletin of the American Mathematical Society}},
    Journal = {Bull. Am. Math. Soc.},
    %ISSN = {0002-9904},
    Volume = {75},
    Pages = {179--239},
    Year = {1969},
    Publisher = {American Mathematical Society (AMS), Providence, RI}
    %Language = {English},
    %DOI = {10.1090/S0002-9904-1969-12129-4},
    %Zbl = {0185.33801}
}

@Article{Smi1970,
    Author = {Smith, K. T.},
    Title = {Formulas to represent functions by their derivatives},
    FJournal = {{Mathematische Annalen}},
    Journal = {Math. Ann.},
    %ISSN = {0025-5831},
    Volume = {188},
    Pages = {53--77},
    Year = {1970},
    Publisher = {Springer, Berlin/Heidelberg}
    %Language = {English},
    %DOI = {10.1007/BF01435415},
    %MSC2010 = {35C15 35B45 46F10},
    %Zbl = {0324.35009}
}
